\documentclass[12pt,reqno]{amsart}
\usepackage[margin=1in]{geometry}
\usepackage{amsmath,amssymb,amsthm,graphicx,amsxtra, setspace}
\usepackage[utf8]{inputenc}
\usepackage{mathrsfs}
\usepackage{hyperref}
\usepackage{xcolor}
\usepackage{upgreek}
\usepackage{mathtools}
\allowdisplaybreaks
\newtheorem{thm}{Theorem}[section]
\newcommand{\norm}[1]{\left\lVert#1\right\rVert}
\newtheorem{lem}{Lemma}[section]
\newtheorem{rem}{Remark}[section]

\newtheorem{Def}{Definition}[section]

\newtheorem{Ex}{Example}[section]
\newtheorem{Ass}{Assumption}[section]

\usepackage{latexsym}
\usepackage{hyperref}
\usepackage{graphicx}
\usepackage{epstopdf}

\allowdisplaybreaks
\let\originalleft\left
\let\originalright\right
\renewcommand{\left}{\mathopen{}\mathclose\bgroup\originalleft}
\renewcommand{\right}{\aftergroup\egroup\originalright}

\makeatletter
\@namedef{subjclassname@2020}{%
	\textup{2020} Mathematics Subject Classification}
\makeatother

\title[Approximate controllability of second order systems]{Approximate controllability of non-autonomous second order impulsive functional evolution equations in Banach spaces}
\author [Sumit Arora, Soniya Singh,  Manil T. Mohan and Jaydev dabas]{}
\email{sonia.iitd.21@gmail.com}
\email{sumit@as.iitr.ac.in}
\email{maniltmohan@ma.iitr.ac.in, maniltmohan@gmail.com}
\email{jay.dabas@gmail.com}
\thanks{$^*$Corresponding author: Jaydev Dabas}

\begin{document}
	\maketitle
	\centerline{\scshape Sumit Arora, Soniya Singh}
	\medskip
	{\footnotesize
		\centerline{Department of Applied Science and Engineering}
		\centerline{Indian Institute of Technology Roorkee,}
		\centerline{Roorkee, Uttarakhand 247667, India.}
	} 
	
	\medskip
	\centerline{\scshape Manil T. Mohan}
	\medskip
	{\footnotesize
		
		\centerline{Department of Mathematics}
		\centerline{Indian Institute of Technology Roorkee,}
		\centerline{Roorkee, Uttarakhand 247667, India.}
	}
	\medskip
	
	\centerline{\scshape Jaydev Dabas$^*$}
	\medskip
	{\footnotesize
		\centerline{Department of Applied Science and Engineering}
		\centerline{Indian Institute of Technology Roorkee,}
		\centerline{Roorkee, Uttarakhand 247667, India.}
	}
	\bigskip
	\begin{abstract}
		This article investigates the approximate controllability of second order non-autonomous functional evolution equations involving non-instantaneous impulses and nonlocal conditions. First, we discuss the approximate controllability of second order linear system in detail, which lacks in the existing literature. Then, we derive sufficient conditions for approximate controllability of our system in separable reflexive Banach spaces via linear evolution operator, resolvent operator conditions, and Schauder's fixed point theorem. Moreover, in this paper, we define proper identification of resolvent operator in Banach spaces. Finally, we verify our results to examine the approximate controllability of the non-autonomous wave equation with non-instantaneous impulses and finite delay in the application section.
	\end{abstract}
	\section{Introduction}\label{intro} \setcounter{equation}{0}
	Controllability is one of the fundamental notions in mathematical control theory and plays a crucial role in various control problems such as the time optimal control problems \cite{VB},  irreducibility of transition semigroups \cite{GDJZ}, stabilization of unstable systems via feedback control \cite{VB1} etc. In the finite  dimensional settings, the problems of exact and approximate controllability are same, whereas in the infinite dimensions, one has to distinguish these two concepts. Exact controllability refers that the solution of a control system can steer from an arbitrary initial state to a desired final state, while the approximate controllability means that the solution enables to steered an arbitrary small neighborhood of a final state. In the infinite dimensional case, the approximately controllable systems are more adequate and have extensive range of applications (cf. \cite{MN,TRR,TR,EZ}, etc). In the past two decades, a good  number of publications discussed the problems of existence and approximate controllability of non-linear evolution systems (in Hilbert and Banach spaces), see for instance, \cite{AM,AJ,FUX,AG,HR,MN,RK,SS}, etc and the references therein.
	
	There are many dynamical systems with nonlocal initial conditions are more suitable than regular initial conditions. For example, the diffusion equation with boundary conditions of  an integral form and  Hydratational heat in which the intensity of heat sources depends on the amount of heat already produced, see \cite{nt}.  Byszewski  and Lakshmikantham \cite{by} first discussed the abstract Cauchy problem with nonlocal initial conditions and established the existence and uniqueness of a mild solution for that problem. In the recent times, several papers appeared in the literature dealing with the approximate controllability problem of the non-linear evolution systems with nonlocal initial conditions, see for instance, \cite{AS,UA,SH1,VV}, etc. On the other hand, there are many physical phenomena in which the current state of a system is influenced by the previous states. These types of phenomena is suitably modeled by delay differential equations, which naturally occurs in ecological models, neural network, inferred grinding models, logistic reaction-diffusion model with delay, etc (cf. \cite{FC,LA,NJ}, etc). 
	
	Many evolutionary processes such as harvesting, shocks, and natural disasters etc, are experience abrupt changes in their states for negligible time instants. Generally, these short-term perturbations are estimated as instantaneous impulses and such processes are mathematically modeled by impulsive differential and partial differential equations. The theory of impulsive evolution equations has found various applications such as threshold, bursting rhythm models in medicine, optimal control models in economics, paramacokinetics and frequency modulated systems, for more details, one can refer the monographs \cite{BM} and \cite{LVB}. In the past few years, many works formulated sufficient conditions for the  approximate controllability of the impulsive systems with instantaneous impulses, see for instance, \cite{AS,SS,RSA,YZ}, etc and the references therein. Moreover, the evolution processes in pharmacotherapy such as the distribution of drugs in the bloodstream and the consequent absorption of the body are a moderate and continuous process and the dynamics of such process can not be interpret by the instantaneous impulsive models.  Therefore, this situation can be analyzed as an impulsive action that starts suddenly and remains active over a finite time interval. Hence, the dynamics of such phenomena is characterized by non-instantaneous impulsive systems. 
	
	Hern\' {a}ndez and O'Regan \cite{HE} first considered a non-instantaneous impulsive abstract differential equation and studied the global solvability of that system. Later, Feckan et al. in \cite{FM} altered the impulsive conditions considered in \cite{HE} and investigate the existence and uniqueness of solutions. In \cite{PM}, Pierri et al. discussed a global solution for a non-instantaneous impulsive evolution equations. Recently, a few developments have been reported on the controllability problems of the non-instantaneous impulsive systems. In \cite{KA,MM}, Muslim et al., considered a nonlinear second order control system with non-instantaneous impulses and studied the existence and stability of the solution and also proved the exact controllability of the considered  system. Ahmed and his co-authors in \cite{AH} examined the approximate controllability of the non-instantaneous impulsive  Hilfer fractional neutral stochastic integro-differential equations with fractional Brownian motion and nonlocal condition.

	In the literature, it has been observed that the problem of approximate controllability for the first and second order autonomous systems are extensively studied. However, the approximate controllability results for the non-autonomous semilinear systems remains limited. There are a few publications available on the approximate controllability of the first and second order non-autonomous semilinear systems, see for  example, \cite{AM,FUX,FUH,KV,MV,RK}, etc.  In \cite{MV}, Mahmudov et al. considered the non-autonomous second order differential inclusions and investigated the approximate controllability results in Hilbert space by applying the Bohnenblust-Karlin's fixed point theorem. Also, they developed approximate controllability for an impulsive system with nonlocal initial conditions. Kumar and his co-authors in \cite{KV} investigated the approximate controllability of a second order non-autonomous system with finite delay in Banach spaces by applying a fixed point approach. The resolvent operator considered in that work is well define only if the state space is a Hilbert space (see, Remark \ref{lem2.1} below). So, it seem to the authors that the results developed in \cite{KV} are reasonable only in separable Hilbert spaces. Moreover, to the best of author's knowledge, the study of the approximate controllability of second order non-autonomous abstract Cauchy problems with non-instantaneous impulses in Banach spaces is leftover in the literature.

	Motivated from the above discussions, in this paper, we formulate sufficient conditions for the approximate controllability of the following non-autonomous second order impulsive evolution equations consist of non-instantaneous impulses and finite delay:
	\begin{subequations}
		\begin{align}
		\label{1.1} {x}''(t)&= \mathrm{A}(t)x(t)+\mathrm{B}u(t)+f(t,x_{t}),\ t\in \bigcup_{i=0}^{N}(s_i,t_{i+1}]\subset J =[0,T],\\
		\label{1} x(t)&=\rho_i(t,x(t_i^{-})),\ t\in(t_i,s_i],\ i=1,\dots,N,\\
		\label{1.3} x'(t)&=\rho'_i(t,x(t_i^{-})), \ t\in(t_i,s_i],\ i=1,\dots,N,\\
		\label{1.4} x(t)&=\phi(t), \ t \in [-q,0), \ q>0, \\ \label{15}  x(0)&=\phi(0)+g(x),\\
		\label{1.5}	x'(0)&=\eta+h(x),\end{align}
	\end{subequations}
	where 
	\begin{itemize}
		\item 	the state variable $x(\cdot)$ takes values in a separable reflexive Banach space $\mathbb{X}$ having strictly convex dual $\mathbb{X}^*$ and $\mathbb{U}$ is a separable Hilbert space, 
		\item the family of linear operators $\{\mathrm{A}(t) : t \in J\}$ on $\mathbb{X}$ is closed and domain is dense in $\mathbb{X}$, \item  the operator $\mathrm{B} : \mathbb{U}\rightarrow	\mathbb{X}$ is  bounded with $\left\|\mathrm{B}\right\|_{\mathcal{L}(\mathbb{U};\mathbb{X})}\le M_B,$ where $\mathcal{L}(\mathbb{U};\mathbb{X})$ denotes the space of bounded linear operators from $\mathbb{U}$ to $\mathbb{X}$, and the control function $u \in  \mathrm{L}^2(J;\mathbb{U)}$,  \item  the nonlinear function $f :\bigcup_{i=0}^{N}[s_i,t_{i+1}]\times \mathcal{D}\rightarrow \mathbb{X},$ (see \eqref{1.2} below for the definition of $\mathcal{D}$),\item  the functions $g,h:\mathrm{PC}(J_q;\mathbb{X})\to\mathbb{X}, J_q=[-q,T]$ (see \eqref{29} below for the definition of the space $\mathrm{PC}(J_q;\mathbb{X})$), \item the fixed points $s_i$ and $t_i$ satisfy  $0=t_0=s_0<t_1\le s_1\le t_2<\ldots<t_N\le s_N\le t_{N+1}=T$ and $x(t_i^+)$ and $x(t_i^-)$ exist for $i = 1,\dots, N$ with $x(t_i^-)=x(t_i)$, \item  for every $t \in J,\ x_t \in \mathcal{D}$ and we define $x_t(s)= x(t + s),\ -q\le s< 0$, \item the functions $\rho_i(t,x(t_i^-)) :[t_i,s_i]\times\mathbb{X}\rightarrow\mathbb{X}$ and their derivatives $\rho'_i(t,x(t_i^-)):[t_i,s_i]\times\mathbb{X}\rightarrow\mathbb{X},$ for $i=1,\dots,N$ represent the non-instantaneous impulses.
	\end{itemize}
	
	The rest of the manuscript is structured as follows: In section \ref{pre}, we review some fundamental definitions, important results and provide the assumptions which is useful to develop our results in subsequent sections. In section \ref{SNAS}, we formulate sufficient condition for the approximate controllability of the non-autonomous system  \eqref{1.1}-\eqref{1.5}. To determine this, we first investigate the approximate controllability of the linear control system corresponding to the problem \eqref{1.1}-\eqref{1.5}. Further, we prove the existence of a mild solution for the system  \eqref{1.1}-\eqref{1.5} via Schauder's fixed point theorem. Then, we establish the approximate controllability of the considered system. In the final section, one notable example of wave equation with non-instantaneous impulses and finite delay is discussed.	
	\section{Preliminaries}\label{pre}\setcounter{equation}{0}
	In present section, we recall some basic definitions, results and assumptions that will be useful in succeeding sections. 
	
	Recently, the study of the non-autonomous abstract second order initial value problem is more extensive. Let us consider the following system:
	\begin{equation}\label{2.3}
	\left\{
	\begin{aligned}
	{x}''(t)&= \mathrm{A}(t)x(t)+f(t),\  0\le s, t\le T,\\
	x(s)&=v, \ {x}'(s)=w,
	\end{aligned}
	\right.
	\end{equation} 
	where the operator $\mathrm{A}(t) : \mathrm{D}(\mathrm{A(t)}) \subseteq \mathbb{X} \rightarrow \mathbb{X}$ is closed and densely defined for each $t\in J$ and the function $f : J \rightarrow \mathbb{X}$ is define suitably. Many works discussed the existence of a solution to the problem \eqref{2.3}, we refer the interested readers to \cite{Bo,MK,Yl,Eo} and the reference therein. In these works, generally the existence of a solution of the system \eqref{2.3} is depends on the existence of an evolution operator $\mathcal{S}(t, s)$ for the equation
	\begin{align}
	{x}''(t)&= \mathrm{A}(t)x(t),\ t\in J.
	\end{align}
	Throughout this work, we consider that the domain of $\mathrm{A}(t)$ is a subspace $\mathrm{D}(\mathrm{A})$, which is independent of $t$ and dense in $\mathbb{X}$. We also assumed the function $t \mapsto \mathrm{A}(t)x$ is continuous for each $x \in \mathrm{D}(\mathrm{A})$. 
	
	\subsection{Evolution operator and cosine family}  
	
	We now introduce the following concept of evolution operator given in \cite{MK}.	
	\begin{Def}\label{def2.1}
		A map $\mathcal{S} : J \times J \rightarrow \mathcal{L}(\mathbb{X};\mathbb{X})=:\mathcal{L}(\mathbb{X})$ said to be an evolution operator if the following conditions are satisfied:
		\begin{itemize}
			\item[(D1)] The map $ (t, s) \rightarrow \mathcal{S}(t, s)x$ is continuously differentiable for each $x\in\mathbb{X}$ and
			\subitem(a) $\mathcal{S} (t, t)=0$ for $t\in J$,
			\subitem(b) for each $ x \in \mathbb{X}$ and $ t,s \in J$,
			\begin{align*}
			\frac{\partial }{\partial t} \mathcal{S}(t, s) x |_{t=s}=x \ \mbox{and} \ \frac{\partial }{\partial s} \mathcal{S}(t, s) x |_{t=s}=-x.
			\end{align*}
			\item[(D2)] If $x\in\mathrm{D}(\mathrm{A})$, then $ \mathcal{S}(t, s)x \in  \mathrm{D}( \mathrm{A})$ for all $t, s \in J$ and the map $(t, s) \mapsto\mathcal{S}(t, s)x$ is of class $C^2$ and
			\subitem(a)$\frac{\partial^2 }{\partial t^2}\mathcal{S}(t, s)x=\mathrm{A}(t)\mathcal{S}(t, s)x$,
			\subitem(b)$\frac{\partial^2 }{\partial s^2}\mathcal{S}(t, s)x=\mathcal{S}(t, s)\mathrm{A}(s)x$,
			\subitem(c)$\frac{\partial^2 }{\partial s \partial t}\mathcal{S}(t, s)x\vert_{t=s}=0$.
			\item[(D3)]For all $t, s \in J$, if $x\in \mathrm{D}(\mathrm{A})$, then $ \frac{\partial }{\partial s}\mathcal{S}(t, s)x \in  \mathrm{D}( \mathrm{A})$, the derivatives  $ \frac{\partial^3 }{\partial t^2\partial s}\mathcal{S}(t, s)x$, $ \frac{\partial^3 }{\partial s^2\partial t}\mathcal{S}(t, s)x $ exist and
			\subitem(a)$ \frac{\partial^3 }{\partial t^2\partial s}\mathcal{S}(t, s)x =\mathrm{A}(t)\frac{\partial }{\partial s}\mathcal{S}(t, s)x$,
			\subitem(b)$ \frac{\partial^3 }{\partial s^2\partial t}\mathcal{S}(t, s)x =\frac{\partial }{\partial t}\mathcal{S}(t, s)\mathrm{A}(s)x$. 
			
			Moreover, the map $(t,s)\mapsto \frac{\partial }{\partial s}\mathcal{S}(t, s)x$ is continuous. 
		\end{itemize}
		
	\end{Def}
	Let us assume that there exists an evolution operator $\mathcal{S}(t,s)$ associated to the operator $\mathrm{A}(t)$.  We define the operator $\mathcal{C}(t, s)=-\frac{\partial }{\partial s}\mathcal{S}(t, s)$ and there exists a set of positive constants $M,\ \tilde{M}$ and $N$ such that 
	\begin{align}
	\sup\limits_{0\le s, t\le T}\norm{\mathcal{C}(t, s)}_{\mathcal{L}(\mathbb{X})}&\le M,\label{23}\\
	\sup\limits_{0\le s, t\le T}\norm{\mathcal{S}(t, s)}_{\mathcal{L}(\mathbb{X})}&\le \tilde{M},\label{24}\\
	\norm{\mathcal{S}(t+\tau, s)-\mathcal{S}(t, s)}_{\mathcal{L}(\mathbb{X})}&\le N |\tau|, \ \mbox{ for all }\ \ {t,t+\tau,s\in J}.\label{25}
	\end{align}
	Moreover, various approach have been discussed in the literature about the existence of the evolution operator $\mathcal{S}(\cdot, \cdot)$ (cf. \cite{Bo,MK,Yl,Eo,Oe,SH}). A very often studied situation is that the operator $\mathrm{A}(t)$ is the perturbation of an operator $\mathrm{A}$, which generates a strongly continuous cosine family. Therefore, it is necessary to review some  properties of the cosine family. 
	
	Let $\mathrm{A} : \mathrm{D}(\mathrm{A}) \subseteq \mathbb{X} \rightarrow \mathbb{X}$ be an infinitesimal generator of a strongly continuous cosine family $\{\mathrm{C}_0(t): t\in\mathbb{R}\}$ of bounded linear operators on $\mathbb{X}$ and the associated sine family $\{\mathrm{S}_0(t) : t \in \mathbb{R}\}$ on $\mathbb{X}$ is defined as
	\begin{align*}
	\mathrm{S}_0(t)x=\int_{0}^{t}\mathrm{C}_0(s)x\mathrm{d}s,\ x\in\mathbb{X},\ t\in\mathbb{R}.
	\end{align*}
	Moreover
	\begin{align*}
	\mathrm{C}_0(t)x-x=\mathrm{A}\int_{0}^{t}\mathrm{S}_0(s)x\mathrm{d}s,\ x\in\mathbb{X},\ t\in\mathbb{R}.
	\end{align*}
	The infinitesimal generator $\mathrm{A}$ of a strongly continuous cosine family  $\{\mathrm{C}_0(t) : t \in \mathbb{R}\}$ is defined as
	\begin{align*}
	\mathrm{A}x&=\frac{\mathrm{d}^2}{\mathrm{d}t^2}\mathrm{C}_0(t)x\Big|_{t=0},\ x\in \mathrm{D}(\mathrm{A}),
	\end{align*}
	where 
	\begin{align*}
	\mathrm{D}(\mathrm{A})&=\{x\in \mathbb{X}:\mathrm{C}_0(t)x\text{ is twice continuously differentiable function in }t \},
	\end{align*}
	equipped with the graph norm
	\begin{align*}
	\norm{x}_{\mathrm{D}(\mathrm{A})}&=\norm{x}_{\mathbb{X}}+\norm{\mathrm{A}x}_{\mathbb{X}},\ x\in \mathrm{D}(\mathrm{A}).
	\end{align*}
	We also define the set
	\begin{align*}
	\mathrm{E}&=\{x\in \mathbb{X}:\mathrm{C}_0(t)x\text{ is once continuously differentiable function of }t \},
	\end{align*}
	endowed with the norm
	\begin{align*}
	\norm{x}_{1}&=\norm{x}_{\mathbb{X}}+\sup\limits_{0\le t\le1}\norm{\mathrm{A}\mathrm{C}_0(t)x}_{\mathbb{X}},\ x\in \mathrm{E},
	\end{align*}
	which forms a Banach space (see \cite{KJ}). Moreover, the operator $\mathcal{A}=\begin{pmatrix}
		0 & \mathrm{I} \\
		\mathrm{A} & 0\\
	\end{pmatrix}$ defined on $\mathrm{D}(\mathrm{A})\times \mathrm{E}$ generate a strongly continuous group of bounded linear operators 
	\begin{align*}
	\mathcal{H}(t)=\begin{pmatrix}
	\mathrm{C}_0(t) & \mathrm{S}_0(t) \\
	\mathrm{A}\mathrm{S}_0(t) & \mathrm{C}_0(t)\\
	\end{pmatrix}
	\end{align*}
  on the space $\mathrm{E}\times\mathbb{X},$ (see Proposition 2.6, \cite{CT}). From this fact, it follows that $\mathrm{A}\mathrm{S}_0(t) : \mathrm{E} \rightarrow \mathbb{X}$ is a bounded linear operator such that $$\mathrm{A}\mathrm{S}_0(t)x \rightarrow 0\ \text{ as }\ t \rightarrow 0, \ \mbox{ for each }\ x \in \mathrm{E}.$$
	
	Travis and Webb  \cite{CT,TCC,CT1} discussed existence results of the following  second order abstract Cauchy problem:
	\begin{equation}\label{2.7}
	\left\{
	\begin{aligned}
	{x}''(t)&= \mathrm{A}x(t)+f(t),\  t\in J,\\
	x(s)&=v, \ {x}'(s)=w.
	\end{aligned}
	\right.
	\end{equation} 
	If the function $f : J \rightarrow \mathbb{X}$ is integrable, then a continuous function $x:[0,T]\to\mathbb{X}$  given by
	\begin{align}\label{2.8}
	x(t)&=\mathrm{C}_0(t-s)v+\mathrm{S}_0(t-s)w+\int_{s}^{t}\mathrm{S}_0(t-\tau)f(\tau)\mathrm{d}\tau,
	\end{align}
	is called a \emph{mild solution} of \eqref{2.7}. Moreover, when $v \in \mathrm{E}$, the function $x(\cdot)$ is continuously differentiable and
	\begin{align*}
	x'(t)&=\mathrm{A}\mathrm{S}_0(t-s)v+\mathrm{C}_0(t-s)w+\int_{s}^{t}\mathrm{C}_0(t-\tau)f(\tau)\mathrm{d}\tau.
	\end{align*}
	Such a solution \eqref{2.8} is called a \emph{strong solution}. Furthermore, if $v \in \mathrm{D}(\mathrm{A}),\ w \in \mathrm{E}$ and $ f$ is a continuously differentiable function, then the function $x(\cdot)$ given in \eqref{2.8} becomes a \emph{classical solution} of the initial value problem \eqref{2.7}.
	
	Now, we assume that $\mathrm{A}(t) = \mathrm{A} + \mathrm{F}(t),$ where $\mathrm{F} : \mathbb{R} \rightarrow \mathcal{L}(\mathrm{E};\mathbb{X})$ is a map such that the function $t \mapsto \mathrm{F}(t)x$ is continuously differentiable in $\mathbb{X}$ for each $x\in \mathrm{E}$. Serizawa \cite{SH} established that, if $v \in\mathrm{D}(\mathrm{A}),\ w \in\mathrm{E},$ the following non-autonomous system 
	\begin{equation}\label{2.9}
	\left\{
	\begin{aligned}
	{x}''(t)&= (\mathrm{A}+\mathrm{F}(t))x(t)+f(t),\  t\in J,\\
	x(0)&=v, \ {x}'(0)=w,
	\end{aligned}
	\right.
	\end{equation} 
	has a unique solution $x(\cdot)$ such that the function $t \mapsto x(t)$ is continuously differentiable in $\mathrm{E}$. The same argument follows to conclude that the system \eqref{2.9} along with the initial condition given in the problem  \eqref{2.7} has a unique solution $x(\cdot, s)$ such that the function $t\mapsto x(t, s)$ is continuously differentiable in $\mathrm{E}$. From \eqref{2.8}, we infer that the solution $x(t,s)$ can be written as 
		\begin{align}\label{2.10}
	x(t,s)&=\mathrm{C}_0(t-s)v+\mathrm{S}_0(t-s)w+\int_{s}^{t}\mathrm{S}_0(t-\tau)\mathrm{F}(\tau)x(\tau,s)\mathrm{d}\tau.
	\end{align}
	Particularly, for $v = 0$, we have
	\begin{align*}
	x(t,s)&=\mathrm{S}_0(t-s)w+\int_{s}^{t}\mathrm{S}_0(t-\tau)\mathrm{F}(\tau)x(\tau,s)\mathrm{d}\tau.
	\end{align*}
	Therefore,
	\begin{align*}
	\norm{x(t,s)}_{1}&\le\norm{\mathrm{S}_0(t-s)}_{\mathcal{L}(\mathbb{X};\mathrm{E})}\norm{w}_{\mathbb{X}}+\int_{s}^{t}\norm{\mathrm{S}_0(t-\tau)}_{\mathcal{L}(\mathbb{X};\mathrm{E})}\norm{\mathrm{F}(\tau)}_{\mathcal{L}(\mathrm{E};\mathbb{X})}\norm{x(\tau,s)}_{1}\mathrm{d}\tau,
	\end{align*} and by using the Gronwall-Bellman lemma, we obtain
	\begin{align}\label{2.11}
	\norm{x(t,s)}_{1}&\le C\norm{w}_{\mathrm{E}}.	
	\end{align}
	Let us now define the operator $$\mathcal{S}(t, s)w = x(t, s).$$  The estimate \eqref{2.11} guarantees that the operator $\mathcal{S}(\cdot,\cdot)$ is bounded on $\mathrm{E}$. Since $\mathrm{E}$ is dense in $\mathbb{X}$, the linear operator $\mathcal{S}(\cdot,\cdot)$ can be extended to $\mathbb{X}$ and we still denoted it by $\mathcal{S}(\cdot,\cdot)$ itself. 
	
	\begin{thm}[Theorem 1.2, \cite{HR}]\label{th2.1}
		Under the preceding conditions on $\mathrm{A}$ and $\mathrm{F}(\cdot)$, $\mathcal{S}(\cdot,\cdot)$ is an evolution	operator for the system \eqref{2.9}. Moreover, if the sine family $\mathrm{S}_0(t)$ is compact for all $t \in J$ implies that the evolution operator  $\mathcal{S}(t, s)$ is also compact	for all $0\leq s \le t\leq T$.
	\end{thm}
	
	\subsection{Resolvent operator and assumptions }To study the approximate controllability of the system \eqref{1.1}-\eqref{1.5}, we define the following operators:
	\begin{equation}\label{ROP}
	\left\{
	\begin{aligned}
	\mathrm{L}_{T}u &:= \int_{0}^{T}\mathcal{S}(T, t) \mathrm{B}u(t) \mathrm{d}t,  \\ 
	\Psi_{0}^{{T}} &:= \int_{0}^{T}\mathcal{S}(T, t) \mathrm{B}\mathrm{B}^{*}\mathcal{S}(T, t)^{*} \mathrm{d}t= \mathrm{L}_{T}(\mathrm{L}_{T}) ^{*},\\ 
	\mathrm{R}(\lambda,\Psi_{0}^{{T}})&:=(\lambda \mathrm{I}+\Psi_{0}^{{T}}\mathcal{J})^{-1}, \ \lambda>0,\ 
	\end{aligned}
	\right.
	\end{equation} 
	where $\mathrm{B}^{*} $ and $\mathcal{S}(T, t)^{*}$ denote the adjoint operators of $\mathrm{B} $ and $\mathcal{S}(T, t) $ respectively.  Moreover, the map $\mathcal{J}$ stands for the duality mapping.
	\begin{Def}\label{def2.3}
		The duality mapping $\mathcal{J}:\mathbb{X}  \rightarrow 2^{\mathbb{X}^*}$ is defined as $$\mathcal{J}=\{x^*\in \mathbb{X}^*:\langle x, x^*\rangle=\norm{x}_{\mathbb{X}}^{2}=\norm{x^*}_{\mathbb{X}^*}^{2}\}, \ \text{ for all } \ x\in \mathbb{X},$$
		where  $\langle \cdot, \cdot  \rangle $  represents  the duality pairing between $\mathbb{X}$ and $\mathbb{X}^*$.  
	\end{Def}
	\begin{rem}\label{lem2.1}
		\begin{itemize}
			\item[(i)] If the space $\mathbb{X}$ is reflexive, then the space $\mathbb{X}$ and $\mathbb{X}^*$ become strictly convex (cf. \cite{VB}). Moreover, the strict convexity of $\mathbb{X}^*$ guarantees that the mapping $\mathcal{J}$ is bijective, strictly monotonic and demicontinuous, that is,
			$$x_{k} \rightarrow x \ \text{ in }	\ {\mathbb{X}} \Rightarrow\mathcal{J}[x_{k}]	\xrightharpoonup{{w}}\mathcal{J}[x] \ \text{ in }\ \mathbb{X}^* \text{ as } k \rightarrow \infty.$$
			\item [(ii)]If $\mathbb{X}$ is a Hilbert space (identified with its own dual), then $\mathcal{J}=\mathrm{I}$, the identity operator on $\mathbb{X}$.
		\end{itemize}
	\end{rem}
	
	Let us define the set 
	\begin{align*} 
	\mathrm{PC}(J;\mathbb{X})&:=\{x:J \rightarrow \mathbb{X} : x\vert_{t\in I_i}\in\mathrm{C}(I_i;\mathbb{X}),\ I_i:=(t_i, t_{i+1}],\ i=0,1,\ldots,N, \ x(t_i^+) \mbox{ and }\ x(t_i^-)\ \\& \qquad \mbox{ exist for each }\ i=1,\ldots,N, \ \mbox{ and satisfy }\ x(t_i)=x(t_i^-)\}, 
	\end{align*}
	 with the norm $\left\|x\right\|_{\mathrm{PC}(J;\mathbb{X})}:=\sup\limits_{t\in J}\left\|x(t)\right\|_{\mathbb{X}}$. Moreover, we also define the set 
	\begin{align}\label{29}
	\mathrm{PC}(J_q;\mathbb{X})&:=\left\{x: J_q\rightarrow \mathbb{X} : x\vert_{t\in [-q,0)}\in\mathcal{D}\ \mbox{and}\ x\vert_{t\in J}\in\mathrm{PC}(J;\mathbb{X})\right\}, 
	\end{align}
	equipped with the norm $\left\|x\right\|_{\mathrm{PC}([-q,T];\mathbb{X})}:=\frac{1}{q}\int_{-q}^{0}\norm{x(s)}_{\mathbb{X}}\mathrm{d}s+\sup\limits_{t\in J}\left\|x(t)\right\|_{\mathbb{X}}$, where
	\begin{align}\label{1.2}
	\mathcal{D}:&= \{\phi : [-q, 0] \rightarrow\mathbb{X}:\phi \ \text{is piecewise continuous with jump discontinuity} \},\end{align}
	endowed with the norm $\norm{\phi}_\mathcal{D}=\frac{1}{q}\int_{-q}^{0}\norm{\phi(s)}_{\mathbb{X}}\mathrm{d}s$ (see\cite{AG}).

	In order to determine the existence and approximate controllability results for the system \eqref{1.1}-\eqref{1.5}, we impose the following assumptions:
	\begin{Ass}\label{as}
		\begin{itemize}
			\item[\textit{(H0)}] for every $y\in\mathbb{X}$, $$z_{\lambda}(y)=\lambda(\lambda I+\Psi_{0}^{\mathrm{T}}(\mathcal{J}))^{-1}(y)\rightarrow 0\ \text{ as }\ \lambda\downarrow 0$$ in strong topology, where $z_{\lambda}(y)$ is a solution of the equation
			\begin{equation}\label{2.15}
			\lambda z_{\lambda}(y)+\Psi_{0}^{\mathrm{T}} \mathcal{J}[z_{\lambda}(y)]=\lambda y.
			\end{equation}
			\item[\textit{(H1)}]$ \mathrm{S}_{0}(t),t\in J$ is compact.
			\item[\textit{(H2)}]
			\begin{enumerate}
				\item[(i)]Let $x:J_q \rightarrow \mathbb{X}$ be such that $x_{0}=\phi$ and $x|_{J}\in \mathrm{PC}(J;\mathbb{X})$. The function 
				$ f:J_1\times\mathcal{D}\to\mathbb{X}$, where $J_1=\bigcup_{i=0}^{N}[s_i,t_{i+1}]$ is strongly measurable in $t$, for each $\phi\in\mathcal{D}$ and continuous in $\phi$, for a.e. $t\in J_1$.
				\item[(ii)] There exists a function $\gamma\in \mathrm{L}^{1}(J_1;\mathbb{R}^{+})$, such that
				$$\|f(t, \phi)\|_{\mathbb{X}}\le\gamma(t), \text{ for a.e.}\ t\in J_1 \text{ and  for all}\ \phi\in \mathcal{D}.$$
			\end{enumerate}
			\item[\textit{(H3)}] The non-instantaneous impulses $\rho_{i}, :[t_{i}, s_{i}]\times\mathbb{X}\rightarrow\mathbb{X}$, for $i=1,\cdots,N$, are such that 
			\begin{itemize}
				\item [(i)] the impulses $\rho_i(\cdot,x), :[t_i,s_i]\to\mathbb{X}$ are continuously differentiable for each $x\in\mathbb{X}$,
				\item [(ii)] for all  $t\in[t_i,s_i]$, the impulses $\rho_i(t,\cdot):\mathbb{X}\to\mathbb{X}$ are completely continuous and their derivatives $\rho'_i(t,\cdot):\mathbb{X}\to\mathbb{X}$ are continuous,
				\item [(iii)] $\norm{\rho_{i}(t,x)}_{\mathbb{X}}\le d_{i},\ \norm{\rho'_{i}(t,x)}_{\mathbb{X}}\le e_{i}, \text{ for all } t\in [t_{i},s_{i}],\ x\in\mathbb{X},\ i=1,\cdots,N$, where $d_{i}$'s and $e_{i}$'s are positive constants.
			\end{itemize}
			\item[\textit{(H4)}] The function $ g,h :\mathrm{PC}(J_q;\mathbb{X})\to \mathbb{X}$ are such that 
			\begin{enumerate}
				\item [(i)] $g$ is completely continuous and $h$ is continuous,
				\item [(ii)] for all $x\in\mathrm{PC}(J_q;\mathbb{X})$, there exist a constants $ M_{g}, M_{h} $ such that $$ \|g(x)\|_{\mathbb{X}}\leq M_{g}\left(\|x\|_{\mathrm{PC}(J_q;\mathbb{X})}+1\right),\ \ \|h(x)\|_{\mathbb{X}}\leq M_{h}\left(\|x\|_{\mathrm{PC}(J_q;\mathbb{X})}+1\right).$$  
			\end{enumerate}
		\end{itemize}
	\end{Ass}
	\begin{rem}\label{rem2.3} Note that the equation \eqref{2.15} has a unique solution $z_{\lambda}(y) = \lambda(\lambda I +\Psi^{T}_{0}\mathcal{J})^{-1}(y)=\lambda\mathrm{R}(\lambda,\Psi^{T}_{0})(y)$ for every $ y\in \mathbb{X}$ and $\lambda > 0$, followed by {Lemma 2.2 \cite{MN}}. Moreover 
		\begin{align}\label{2.16}\norm{z_{\lambda}(y)}_{\mathbb{X}} &= \norm{\mathcal{J}[z_{\lambda}(y)]}_{\mathbb{X}'} \le \norm{y}_{\mathbb{X}}.\end{align}
	\end{rem}

	We now introduce the concept of mild solution for the system \eqref{1.1}-\eqref{1.5} (cf. \cite{Er}). 
	\begin{Def}\label{def2.2}
		A function $x(\cdot ; \phi, \eta, u): J_q \rightarrow \mathbb{X}$ is called a \emph{mild solution} of  \eqref{1.1}-\eqref{1.5}, if $x(t)=\phi(t),\ t\in[-q,0)$ and satisfies the following:
		\begin{eqnarray}\label{DEF2.10}
		x(t)=\left\{
		\begin{aligned}
		& \mathcal{C}(t, 0)[\phi(0)+g(x)] + \mathcal{S}(t,0)[\eta+h(x)] +\int_{0}^{t}\mathcal{S}(t, s)[\mathrm{B}u(s)+f(s,x_{s})] \mathrm{d}s, \  t \in  [0,t_{1} ], \\
		&\rho_{i}(t, x(t_{i}^{-})) ,\ t\in(t_i,s_i],\ i=1,\dots,N,\\
		& \mathcal{C}(t, s_{i})\rho_{i}(s_{i}, x(t_{i}^{-})) + \mathcal{S}(t, s_{i})\rho'_i(s_{i}, x(t_{i}^{-})) \\ &\quad +\int_{s_{i}}^{t}\mathcal{S}(t, s)[\mathrm{B}u(s)+f(s,x_{s})] \mathrm{d}s,\ t\in(s_i,t_{i+1}],\ i=1,\dots,N.
		\end{aligned}
		\right.
		\end{eqnarray} 
	\end{Def}
	\begin{rem}
		Note that a mild solution of the system \eqref{1.1}-\eqref{1.5} satisfies the conditions \eqref{1}, \eqref{1.4} and\eqref{15}. Nevertheless, a mild solution  need not be differentiable at $t\in\cup_{i=1}^{N} (t_{i},s_{i} ]\cup \{0\}$. 
	\end{rem}
	\begin{Def}
		 The system \eqref{1.1}-\eqref{1.5} is said to be approximately controllable on $J$, for any initial function $\phi \in \mathcal{D}$, if the closure of the reachable set is the whole space $\mathbb{X}$, that is, $\overline{\mathcal{K}(T,\phi,\eta)}=\mathbb{X}$, where the reachable set is defined as $$\mathcal{K}(T,\phi,\eta)=\{x(T ; \phi, \eta, u): u(\cdot) \in\mathrm{L}^{2}(J;\mathbb{U})\}.$$
	\end{Def}
	
	\section{Approximate Controllability of the Semilinear Non-Autonomous System}\label{SNAS}\setcounter{equation}{0}
	This section is devoted for investigating the approximate controllability of the system \eqref{1.1}-\eqref{1.5}. A set of sufficient conditions will be obtained by studying the approximate controllability of the linear control problem corresponding to the system \eqref{1.1}-\eqref{1.5}.   
	\subsection{Linear control problem}\label{LCP} 
	To establish the approximate controllability of the linear problem, we first determine the existence of an optimal control by minimizing the cost functional given by
	\begin{align}\label{2.17}
	\mathcal{F}(x,u)&=\norm{x(T)-x_T}^2_{\mathbb{X}}+\lambda\int_{0}^{T}\norm{u(t)}^2_{\mathbb{U}},
	\end{align}
	where $x(\cdot)$ is the unique mild solution of the linear control system:
	\begin{equation}\label{L1}
	\left\{
	\begin{aligned}
	{x}''(t)&= \mathrm{A}(t)x(t)+\mathrm{B}u(t),\ t\in  J =[0,T],\\
	x(0)&=v,\ x'(0) =w,
	\end{aligned}
	\right.
	\end{equation} 
	with the control $u\in\mathrm{L}^{2}(J;\mathbb{U}),\ x_T\in\mathbb{X}\text{ and } \lambda>0$. Since $\mathrm{B}u \in \mathrm{L}^1(J;\mathbb{X})$, the existence of a unique mild solution
	\begin{align}
	x(t)&=\mathcal{C}(t,0)v+\mathcal{S}(t,0)w+\int_{0}^{t}\mathcal{S}(t,s)\mathrm{B}u(s)\mathrm{d}s,\ t\in J,
	\end{align}
	for any $u \in\mathrm{L}^{2}(J;\mathbb{U})=\mathscr{U}_{ad}$, 
	to the system \eqref{L1} is immediate by an application of Theorem 2.2, \cite{HR}. Next, we define the {\em admissible class} $\mathscr{A}_{ad}$ for the system \eqref{L1} as $$\mathscr{A}_{ad}=\{(x,u):x \text{ is a unique mild solution of } \eqref{L1} \text{ with control  }u\in \mathscr{U}_{ad}\}.$$ For a given control $u\in \mathscr{U}_{ad}$, the system \eqref{L1} possesses a unique mild solution, and hence the set $\mathscr{A}_{ad}$ is nonempty.
	
	The existence of an optimal pair minimizing the cost functional \eqref{2.17} is discussed in the next theorem:
	\begin{thm}\label{th1} For given $v,w\in\mathbb{X}$, there exists a unique optimal pair  $(x^0,u^0)\in\mathscr{A}_{\mathrm{ad}}$ of the problem:
		\begin{align}\label{F}
		\min\limits_{(x,u)\in \mathscr{A}_{ad}}\mathcal{F}(x, u).
		\end{align}
	\end{thm} 
	A proof of Theorem \ref{th1} can be obtained by proceeding similarly as in the proof of Theorem 3.1 in \cite{SS}. As the cost functional is convex and the system \eqref{L1} is linear, the optimal control obtained in Theorem \ref{th1} is unique. The  expression for the optimal control in the feedback form is provided by the following lemma:
	\begin{lem}\label{lm2.1}
		Assume that $(x,u)$ is the optimal pair for the problem \eqref{F}. Then the optimal control $u$ is given by
		\begin{align}
		u(t)=\mathrm{B}^*\mathcal{S}^*(T,t)\mathcal{J}\left[\mathrm{R}(\lambda,\Psi_{0}^{\mathrm{T}})\ell(x(\cdot))\right],\ t\in J,
		\end{align}
		where
		\begin{align}
		\nonumber	\ell(x(\cdot))=x_T-\mathcal{C}(T,0)v-\mathcal{S}(T,0)w.
		\end{align}
	\end{lem}
	A proof of the lemma can be determined by proceeding similarly as in the proof of Lemma 3.1, \cite{RK}.	In next lemma, we discuss the approximate controllability of the linear control problem \eqref{L1}.
	\begin{lem}\label{lm}
		The linear control system \eqref{L1} is approximately controllable on $J$ if and only if the Assumption  (H0) holds.
	\end{lem}
	\begin{rem}\label{rm}
		From Theorem 2.3 in \cite{MN}, we know that the Assumption (H0) holds implies that the operator $\Psi_0^T$ is positive and vice versa. The positivity of $\Psi_0^T$ is equivalent to say that $$\langle x^*,\Psi_0^Tx^*\rangle=0\Rightarrow x^*=0.$$ Since, we have $$\langle x^*,\Psi_0^Tx^*\rangle=(\mathrm{L}_T^*x^*,\mathrm{L}_T^*x^*)_{\mathrm{L}^2(J;\mathbb{U})}=\norm{\mathrm{L}_T^*x^*}^2_{\mathrm{L}^2(J;\mathbb{U})},$$
		where  $\mathrm{L}_T^*x^*=\mathrm{B}^*\mathcal{S}(T,t)^*x^*,\ t\in J$. Hence by the above fact and Lemma \ref{lm}, one can  ensure that the approximate controllability of the linear system \eqref{L1} is equivalent to the condition $$\mathrm{B}^*\mathcal{S}(T,t)^*x^*=0\ , t\in J, \ \Rightarrow x^*=0.$$ 
	\end{rem}
	\subsection{Approximate controllability of the semilinear system} In this subsection, we establish the approximate controllability of the system \eqref{1.1}-\eqref{1.5}. To achieve this goal, we prove the existence of a mild solution for $\lambda> 0,\ x_T \in \mathbb{X}$ of  the  system \eqref{1.1}-\eqref{1.5} with the control 
	\begin{align}\label{4.1}
	u_{\lambda}(t)=&\sum\limits_{i=0}^{N}u_{i,\lambda}(t)\chi_{[s_i,t_{i+1}]}(t),\  t\in J,
	\end{align}
	where \begin{align*}
	u_{i,\lambda}(t)=\mathrm{B}^*\mathcal{S}^*(t_{i+1},t)\mathcal{J}\left[\mathrm{R}(\lambda,\Psi_{s_i}^{t_{i+1}})g_i(x(\cdot))\right],
	\end{align*}
	for $\ t\in[s_i,t_{i+1}],\ i=0,1,\dots,N$ with
	\begin{align*}
	g_0(x(\cdot))&=x_{T}-\mathcal{C}(t_1, 0)[\phi(0)+g(\tilde{x})] - \mathcal{S}(t_1,0)[\eta+h(\tilde{x})] -\int_{0}^{t_1}\mathcal{S}(t_1, s)f(s,\tilde{x}_{s})\mathrm{d}s,\\
	g_i(x(\cdot))&=x_{T}-\mathcal{C}(t_{i+1}, s_i)\rho_{i}(s_{i}, \tilde{x}(t_{i}^{-}))  - \mathcal{S}(t_{i+1},s_i)\rho'_i(s_{i}, \tilde{x}(t_{i}^{-})) \\&\quad -\int_{s_i}^{t_{i+1}}\mathcal{S}(t_{i+1}, s)f(s,\tilde{x}_{s})\mathrm{d}s, \ i=1,\ldots,N,
	\end{align*}
	and $\tilde{x}: J_q \rightarrow \mathbb{X}$ such that $\tilde{x}(t) = \phi(t),\ t \in [-q,0)$ and $\tilde{x}(t) = x(t),\ t \in J$.
	\begin{rem}
		Since the operator 	$\Psi_{s_i}^{t_{i+1}},$ for each $i = 0, 1,\dots,N$ is non-negative, linear and bounded,  Remark \ref{rem2.3} is also valid for each $\Psi_{s_i}^{t_{i+1}},$ for $i = 0, 1,\dots, N$.
	\end{rem}
	\begin{thm}\label{th3.1}
		If Assumptions (H1)-(H4) hold true, then for every $\lambda > 0$ and fixed $x_T \in \mathbb{X},$ the system \eqref{1.1}-\eqref{1.5} with the control \eqref{4.1} has at least one mild solution on $J,$ provided 
		\begin{align}\label{4.18}
		K\left[1+\frac{(\tilde{M}M_B)^2T}{\lambda}\right]<1,
		\end{align}
		where $K=MM_g+\tilde{M}M_h$.
	\end{thm}
	\begin{proof}
		Let $\mathcal{Z}:= \{x \in\mathrm{PC}(J;\mathbb{X}) : x(0) = \phi(0)+g(x)\}$ be the space endowed with the norm $\norm{\cdot}_{\mathrm{PC}(J;\mathbb{X})}$. We consider a set
		$$\mathcal{E}_r = \{x \in \mathcal{Z} : \norm{\cdot}_{\mathrm{PC}(J:\mathbb{X})}\le r\}$$ where $r$ is a positive constant. For $\lambda > 0$, let us define an operator $\Phi_{\lambda} : \mathcal{Z}\rightarrow \mathcal{Z}$ as
		$$(\Phi_{\lambda} x)(t) = z(t),$$
		where
		\begin{eqnarray}\label{operator}
		z(t)=\left\{
		\begin{aligned}
		& \mathcal{C}(t, 0)[\phi(0)+g(\tilde{x})] + \mathcal{S}(t,0)[\eta+h(\tilde{x})] +\int_{0}^{t}\mathcal{S}(t, s)[\mathrm{B}u_{\lambda}(s)+f(s,\tilde{x}_{s})] \mathrm{d}s, \  t \in  [0,t_{1} ], \\
		&\rho_{i}(t, \tilde{x}(t_{i}^{-})) ,\ t \in (t_{i},s_{i}],\ i=1,\ldots,N,\\
		& \mathcal{C}(t, s_{i})\rho_{i}(s_{i}, \tilde{x}(t_{i}^{-})) + \mathcal{S}(t, s_{i})\rho'_i(s_{i}, \tilde{x}(t_{i}^{-})) \\ &\quad +\int_{s_{i}}^{t}\mathcal{S}(t,s)[\mathrm{B}u_{\lambda}(s)+f(s,\tilde{x}_{s})] \mathrm{d}s, \ t \in(s_{i},t_{i+1} ],\ i=1,\ldots,N,
		\end{aligned}
		\right.
		\end{eqnarray} 
		with $u_\lambda$ is defined in \eqref{4.1} and $\tilde{x}(t)=\phi(t), t\in[-q,0)$ and $\tilde{x}(t)=x(t),t\in J$. By the definition of $\Phi_{\lambda}$, one can ensure that the problem of finding a mild solution of the system \eqref{1.1}-\eqref{1.5} is equivalent to finding a fixed point of the operator $\Phi_{\lambda}$. A proof that the operator $\Phi_{\lambda}$ has a fixed point is divided into the following steps. 
		
		\vskip 0.1in 
		\noindent\textbf{Step (1): }  \emph{$\Phi_\lambda(\mathcal{E}_r )\subseteq\mathcal{E}_r,$ for some $r$}. In contrast, we assume that our claim is not true. Then for any $\lambda > 0$ and for every $r > 0$, there exists $x^r(\cdot) \in \mathcal{E}_r$ such that $\norm{(\Phi_\lambda x^r)(t)}_{\mathbb{X}} > r$, for some $t\in J$ ($t$ may depend upon $r$). Taking  $t \in [0, t_1]$ and using the estimate \eqref{2.16}, Assumption \ref{as} \textit{(H2)}-\textit{(H4)}, we calculate
		\begin{align}\label{4.19}
		\nonumber	r&<\norm{(\Phi_\lambda x^r)(t)}_{\mathbb{X}} \\ \nonumber &= \norm{\mathcal{C}(t, 0)[\phi(0)+g(\tilde{x})] + \mathcal{S}(t,0)[\eta+h(\tilde{x})] +\int_{0}^{t}\mathcal{S}(t, s)[\mathrm{B}u_{0,\lambda}(s)+f(s,\tilde{x}_{s})] \mathrm{d}s}_{\mathbb{X}}\\ \nonumber &\le \norm{\mathcal{C}(t, 0)[\phi(0)+g(\tilde{x})]}_{\mathbb{X}}+\norm{\mathcal{S}(t,0)[\eta+h(\tilde{x})]}_{\mathbb{X}}+\int_{0}^{t}\norm{\mathcal{S}(t, s)[\mathrm{B}u_{0,\lambda}(s)+f(s,\tilde{x}_{s})]}_{\mathbb{X}}\mathrm{d}s\\ \nonumber &\le M[\norm{\phi(0)}_{\mathbb{X}}+M_g(\norm{\tilde{x}}_{\mathrm{PC}(J_q;\mathbb{X})}+1)]+\tilde{M}[\norm{\eta}_{\mathbb{X}}+M_{h}(\norm{x}_{\mathrm{PC}(J_q;\mathbb{X})}+1)]\\\nonumber&\quad+\tilde{M}M_B\int_{0}^{t}\norm{u_{0,\lambda}(s)}_{\mathbb{U}}\mathrm{d}s+\tilde{M}\int_{0}^{t}\norm{f(s,\tilde{x}_{s})}_{\mathbb{X}}\mathrm{d}s\\ \nonumber &\le M[\norm{\phi(0)}_{\mathbb{X}}+M_g(r_1+1)]+\tilde{M}[\norm{\eta}_{\mathbb{X}}+M_{h}(r_1+1)]+\frac{(\tilde{M}M_B)^2T}{\lambda}\norm{g_0(x(\cdot))}_{\mathbb{X}}\\ \nonumber&\quad+\tilde{M}\int_{0}^{t}\gamma(s)\mathrm{d}s\\&\le M\norm{\phi(0)}_{\mathbb{X}}+\tilde{M}\norm{\eta}_{\mathbb{X}}+K(r_1+1)+\frac{(\tilde{M}M_B)^2T}{\lambda}\left[\tilde{K}+K(r_1+1)+\tilde{M}\norm{\gamma}_{\mathrm{L}^1(J_1;\mathbb{X})}\right]\nonumber\\&\quad+\tilde{M}\norm{\gamma}_{\mathrm{L}^1(J_1;\mathbb{X})},
		\end{align}
		where $K=MM_g+\tilde{M}M_h,\ \tilde{K}=\norm{x_T}_{\mathbb{X}}+M\norm{\phi(0)}_{\mathbb{X}}+\tilde{M}\norm{\eta}_{\mathbb{X}}$ and  $r_1=r+\norm{\phi}_{\mathcal{D}}$. For $t \in (t_i, s_i], \ i = 1,\dots N$, we estimate
		\begin{align}\label{4.20}
		r<\norm{(\Phi_\lambda x^r)(t)}_{\mathbb{X}}=\norm{\rho_{i}(t, \tilde{x}(t_{i}^{-}))}_{\mathbb{X}}\le d_i\le d_i+Kr_1+\frac{(\tilde{M}M_B)^2Kr_1T}{\lambda}.
		\end{align}
		Taking $t\in(s_i,t_{i+1}],\ i=1,\dots, N,$ we evaluate
		\begin{align}\label{4.21}
		\nonumber	r &<\norm{(\Phi_\lambda x^r)(t)}_{\mathbb{X}} \\\nonumber&=\norm{\mathcal{C}(t, s_{i})\rho_{i}(s_{i}, \tilde{x}(t_{i}^{-})) + \mathcal{S}(t, s_{i})\rho'_i(s_{i}, \tilde{x}(t_{i}^{-}))   +\int_{s_{i}}^{t}\mathcal{S}(t,s)[\mathrm{B}u_{i,\lambda}(s)+f(s,\tilde{x}_{s})] \mathrm{d}s}_{\mathbb{X}}
		\\\nonumber &\le \norm{\mathcal{C}(t, s_{i})\rho_{i}(s_{i}, \tilde{x}(t_{i}^{-}))}_{\mathbb{X}}+\norm{\mathcal{S}(t, s_{i})\rho'_i(s_{i}, \tilde{x}(t_{i}^{-}))}_{\mathbb{X}}+\int_{s_{i}}^{t}\norm{\mathcal{S}(t,s)\mathrm{B}u_{i,\lambda}(s)}_{\mathbb{X}}\mathrm{d}s\\\nonumber &\quad+\int_{s_{i}}^{t}\norm{\mathcal{S}(t, s)f(s,\tilde{x}_{s})}_{\mathbb{X}}\mathrm{d}s
		\\ \nonumber &\le Md_i+\tilde{M}e_i+\tilde{M}M_B\int_{s_{i}}^{t}\norm{u_{i,\lambda}(s)}_\mathbb{U}\mathrm{d}s+\tilde{M}\int_{s_{i}}^{t}\norm{f(s,\tilde{x}_{s})}_{\mathbb{X}}\mathrm{d}s\\ \nonumber&\le Md_i+\tilde{M}e_i+\frac{(\tilde{M}M_B)^2T}{\lambda}\norm{g_i(x(\cdot))}_{\mathbb{X}}+\tilde{M}\int_{s_{i}}^{t}\gamma(s)\mathrm{d}s
		\\&\le Md_i+\tilde{M}e_i+\frac{(\tilde{M}M_B)^2T}{\lambda}\left[\norm{x_T}_{\mathbb{X}}+Md_i+\tilde{M}e_i+\tilde{M}\norm{\gamma}_{\mathrm{L}^1(J_1;\mathbb{R}^+)}\right]+\tilde{M}\norm{\gamma}_{\mathrm{L}^1(J_1;\mathbb{R}^+)}\nonumber\\&\le Md_i+\tilde{M}e_i+\frac{(\tilde{M}M_B)^2T}{\lambda}\left[K_i+\tilde{M}\norm{\gamma}_{\mathrm{L}^1(J_1;\mathbb{R}^+)}\right]+\tilde{M}\norm{\gamma}_{\mathrm{L}^1(J_1;\mathbb{R}^+)}\nonumber\\&\le Md_i+\tilde{M}e_i+Kr_1+\frac{(\tilde{M}M_B)^2T}{\lambda}\left[K_i+Kr_1+\tilde{M}\norm{\gamma}_{\mathrm{L}^1(J_1;\mathbb{R}^+)}\right]+\tilde{M}\norm{\gamma}_{\mathrm{L}^1(J_1;\mathbb{R}^+)},
		\end{align}
		where $K_i=\norm{x_T}_{\mathbb{X}}+Md_i+\tilde{M}e_i$ for $i=1,\ldots,N.$ Thus, dividing by $r$ in the expressions \eqref{4.19}, \eqref{4.20} and \eqref{4.21}, and then passing $r \rightarrow\infty$, we obtain
		\begin{align}
		K\left[1+\frac{(\tilde{M}M_B)^2T}{\lambda}\right]>1,
		\end{align}
		which is a contradiction to \eqref{4.18}. Hence, for some $r > 0, \ \Phi_{\lambda}(\mathcal{E}_{r})\subseteq \mathcal{E}_{r}.$ 
		\vskip 0.1in 
		\noindent\textbf{Step (2): } \emph{The operator $\Phi_\lambda$ is continuous}. For this, we take a sequence  $\{x^n\}_{n=1}^\infty\subseteq\mathcal{E}_{r}$ such that
		$x^n\rightarrow x$ in $\mathcal{E}_{r}$, that is,
		$$\lim\limits_{n\rightarrow\infty}\norm{x^n-x}_{\mathrm{PC}(J;\mathbb{X})}=0.$$
		For $s\in J$, we estimate
		\begin{align}\label{eqq1}
		\left\|\tilde{x_{s}^n}-\tilde{x_{s}}\right\|_{\mathcal{D}}&=\frac{1}{q}\int_{-q}^{0}\left\|\tilde{x_{s}^{n}}(\theta)-\tilde{x_{s}}(\theta)\right\|_{\mathbb{X}}\mathrm{d}\theta=\frac{1}{q}\int_{-q}^{0}\left\|\tilde{x^{n}}(s+\theta)-\tilde{x}(s+\theta)\right\|_{\mathbb{X}}\mathrm{d}\theta\nonumber\\&=\frac{1}{q}\int_{s-q}^{s}\left\|\tilde{x^{n}}(\tau)-\tilde{x}(\tau)\right\|_{\mathbb{X}}\mathrm{d}\tau.
		\end{align}
		If $s<q$, then  one can write the above expression as
		\begin{align*}
		\norm{\tilde{x_{s}^n}-\tilde{x_{s}}}_{\mathcal{D}}&=\frac{1}{q}\int_{s-q}^{0}\left\|\tilde{x^{n}}(\tau)-\tilde{x}(\tau)\right\|_{\mathbb{X}}\mathrm{d}\tau+\frac{1}{q}\int_{0}^{s}\left\|\tilde{x^{n}}(\tau)-\tilde{x}(\tau)\right\|_{\mathbb{X}}\mathrm{d}\tau\nonumber\\&\le\frac{1}{q}\int_{-q}^{0}\left\|\tilde{x^{n}}(\tau)-\tilde{x}(\tau)\right\|_{\mathbb{X}}\mathrm{d}r+\frac{1}{q}\int_{0}^{T}\left\|\tilde{x^{n}}(\tau)-\tilde{x}(\tau)\right\|_{\mathbb{X}}\mathrm{d}\tau\nonumber\\&=\frac{1}{q}\int_{0}^{T}\left\|x^{n}(\tau)-x(\tau)\right\|_{\mathbb{X}}\mathrm{d}\tau\le\frac{T}{q}\left\|x^n-x\right\|_{\mathrm{PC}(J;\mathbb{X})}\to 0,\ \mbox{ as }\ n\to\infty.
		\end{align*}
		If $s\ge q$, then by the expression \eqref{eqq1}, we obtain
		\begin{align*}
		\left\|\tilde{x_{s}^n}-\tilde{x_{s}}\right\|_{\mathcal{D}}&\le\frac{1}{q}\int_{0}^{s}\left\|\tilde{x^{n}}(\tau)-\tilde{x}(\tau)\right\|_{\mathbb{X}}\mathrm{d}\tau\le\frac{1}{q}\int_{0}^{T}\left\|x^{n}(\tau)-x(\tau)\right\|_{\mathbb{X}}\mathrm{d}\tau\\&\le \frac{T}{q}\left\|x^n-x\right\|_{\mathrm{PC}(J;\mathbb{X})}\to 0,\ \mbox{ as }\ n\to\infty.
		\end{align*}
		By using the  above convergences along with Assumption \ref{as} \textit{(H2)}, we immediately have 
		\begin{align}\label{3.14}
		\norm{f(s,\tilde{x^n_s})-f(s,\tilde{x}_{s})}_{\mathbb{X}}\to 0\ \mbox{ as }\ n\to\infty, \ \mbox{ uniformly for } \ s\in J. 
		\end{align}
		From the convergence \eqref{3.14}, Assumption \ref{as} (\textit{H2}),(\textit{H4}) and the dominated convergence theorem, we obtain
		\begin{align}\label{4.26}
		\norm{g_0(x^n(\cdot))-g_0(x(\cdot))}_{\mathbb{X}}&\le\norm{\mathcal{C}(t_1, 0)[g(\tilde{x^n})-g(\tilde{x})]}_{\mathbb{X}} + \norm{\mathcal{S}(t_1,0)[h(\tilde{x^n})-h(\tilde{x})]}_{\mathbb{X}}\nonumber\\&\quad+ \int_{0}^{t_1}\norm{\mathcal{S}(t_1, s)[f(s,\tilde{x^n_s})-f(s,\tilde{x}_{s})]}_{\mathbb{X}}\mathrm{d}s\nonumber\\&\le M\norm{g(\tilde{x^n})-g(\tilde{x})}_{\mathbb{X}}+\tilde{M}\norm{h(\tilde{x^n})-h(\tilde{x})}_{\mathbb{X}}\nonumber\\&\quad+M\int_{0}^{t_1}\norm{f(s,\tilde{x^n_s})-f(s,\tilde{x}_{s})}_{\mathbb{X}}\mathrm{d}s\nonumber\\&\rightarrow 0\  \text{ as }\  n\rightarrow\infty.
		\end{align}
		Similarly, for $i=1,\ldots,N$, we estimate
		\begin{align}\label{4.27}
		\nonumber&\norm{g_i(x^n(\cdot))-g_i(x(\cdot))}_{\mathbb{X}}\\\nonumber&\le \norm{\mathcal{C}(t_{i+1}, s_i)[\rho_{i}(s_{i}, \tilde{x^n}(t_{i}^{-}))-\rho_{i}(s_{i}, \tilde{x}(t_{i}^{-}))]}_{\mathbb{X}} + \norm{\mathcal{S}(t_{i+1},s_i)[\rho'_i(s_{i}, \tilde{x^n}(t_{i}^{-}))-\rho'_i(s_{i}, \tilde{x}(t_{i}^{-}))]}_{\mathbb{X}} \\\nonumber &\quad+\int_{s_i}^{t_{i+1}}\norm{\mathcal{S}(t_{i+1}, s)[f(s,\tilde{x^n_s})-f(s,\tilde{x}_{s})]}_{\mathbb{X}}\mathrm{d}s\\
		\nonumber &\le M\norm{\rho_{i}(s_{i}, \tilde{x^n}(t_{i}^{-}))-\rho_{i}(s_{i}, \tilde{x}(t_{i}^{-}))}_{\mathbb{X}}+\tilde{M}\norm{\rho'_i(s_{i}, \tilde{x^n}(t_{i}^{-}))-\rho'_i(s_{i}, \tilde{x}(t_{i}^{-}))}_{\mathbb{X}}\\\nonumber &\quad+\tilde{M}\int_{s_i}^{t_{i+1}}\norm{f(s,\tilde{x^n_s})-f(s,\tilde{x}_{s})}_{\mathbb{X}}\mathrm{d}s
		\\&\rightarrow 0\  \text{ as }\ n\rightarrow\infty,
		\end{align}
		where we used the convergence \eqref{3.14}, Assumption \ref{as} \textit{(H2)}-\textit{(H3)} and the dominated convergence theorem. By the convergences \eqref{4.26}, \eqref{4.27} and estimate \eqref{2.16}, we calculate
		\begin{align}
		\nonumber\norm{\mathrm{R}(\lambda,\Psi_{s_i}^{t_{i+1}})g_i(x^n(\cdot))-\mathrm{R}(\lambda,\Psi_{s_i}^{t_{i+1}})g_i(x(\cdot))}_{\mathbb{X}}
		&=\frac{1}{\lambda}\norm{\lambda\mathrm{R}(\lambda,\Psi_{s_i}^{t_{i+1}})\left[g_i(x^n(\cdot))-g_i(x(\cdot))\right] }_{\mathbb{X}}\\\nonumber&\le\frac{1}{\lambda}\norm{g_i(x^n(\cdot))-g_i(x(\cdot))}_{\mathbb{X}}\\
		&\rightarrow 0 \ \text{ as }\ n\rightarrow\infty, \ \text{ for }\ i=0,1,\ldots, N\nonumber.
		\end{align}
		Since the mapping $\mathcal{J}:\mathbb{X}\rightarrow\mathbb{X}$ is demicontinuous, we have
		\begin{align}\label{eq430}
		\mathcal{J}\left[\mathrm{R}(\lambda,\Psi_{s_i}^{t_{i+1}})g_i(x^n(\cdot))\right]\xrightharpoonup{{w}}\mathcal{J}\left[\mathrm{R}(\lambda,\Psi_{s_i}^{t_{i+1}})g_i(x(\cdot))\right]\ \text{ as }\ n\rightarrow\infty\ \text{ in }\ \mathbb{X}^*,
		\end{align}
		for $i=0,1,\ldots,N$. Since by Theorem \ref{th2.1}, we know that the operator $\mathcal{S}(t, s)$ is compact for all $s \le t$. Therefore, the operator $\mathcal{S}(t, s)^*$ is also compact. Hence, by using the compactness of that operator along with the weak convergence \eqref{eq430}, we obtain 
		\begin{align}\label{4.31}
		\nonumber\norm{u_{i,\lambda}^n(t)-u_{i,\lambda}(t)}_{\mathbb{U}} & \le\norm{\mathrm{B}^*\mathcal{S}(t_{i+1},t)^*\left[\mathcal{J}\left[\mathrm{R}(\lambda,\Psi_{s_i}^{t_{i+1}})g_i(x^n(\cdot))\right]-\mathcal{J}\left[\mathrm{R}(\lambda,\Psi_{s_i}^{t_{i+1}})g_i(x(\cdot))\right]\right]}_{\mathbb{U}}\\ \nonumber & \le M_B\norm{\mathcal{S}(t_{i+1},t)^*\left[\mathcal{J}\left[\mathrm{R}(\lambda,\Psi_{s_i}^{t_{i+1}})g_i(x^n(\cdot))\right]-\mathcal{J}\left[\mathrm{R}(\lambda,\Psi_{s_i}^{t_{i+1}})g_i(x(\cdot))\right]\right]}_{\mathbb{X}}\\
		&\rightarrow 0\ \text{ as }\ n\rightarrow\infty, \ \text{ for }\ t\in(s_i,t_{i+1}],\ i=0,1,\ldots,N.
		\end{align}
		Using the convergences \eqref{3.14}, \eqref{4.31} and the dominated convergence theorem, we compute
		\begin{align*}
		&\nonumber\norm{(\Phi_\lambda x^n)(t)-(\Phi_\lambda x)(t)}_{\mathbb{X}}\\\nonumber&\le\int_{0}^{t}\norm{\mathcal{S}(t,s)[\mathrm{B}u_{0,\lambda}^n(s)-\mathrm{B}u_{0,\lambda} (s)]}_{\mathbb{X}}\mathrm{d}s +\int_{0}^{t}\norm{\mathcal{S}(t, s)[f(s,\tilde{x^n_s})-f(s,\tilde{x}_{s})]}_{\mathbb{X}}\mathrm{d}s\\\nonumber&\le
		\tilde{M}M_B \int_{0}^{t}\norm{u_{0,\lambda}^n(s)-u_{0,\lambda}(s)}_{\mathbb{U}}\mathrm{d}s+\tilde{M}\int_{0}^{t}\norm{f(s,\tilde{x^n_s})-f(s,\tilde{x}_{s})}_{\mathbb{X}} \mathrm{d}s\\&\rightarrow 0 \ \text{ as } \ n\rightarrow\infty, \text{ for } \ t\in[0,t_1].
		\end{align*}
		Similarly, for $t\in (s_i,t_{i+1}],\ i=1,\ldots,N,$ we calculate
		\begin{align*}
		&\nonumber\norm{(\Phi_\lambda x^n)(t)-(\Phi_\lambda x)(t)}_{\mathbb{X}}\\\nonumber&\le\norm{\mathcal{C}(t, s_{i})\left(\rho_{i}(s_{i}, \tilde{x^n}(t_{i}^{-})-\rho_{i}(s_{i}, \tilde{x}(t_{i}^{-}))\right)}_{\mathbb{X}} + \norm{\mathcal{S}(t, s_{i})\left(\rho'_i(s_{i}, \tilde{x^n}(t_{i}^{-}))-\rho'_i(s_{i}, \tilde{x}(t_{i}^{-}))\right)}_{\mathbb{X}}\\\nonumber&\quad+\int_{s_{i}}^{t}\norm{\mathcal{S}(t, s)\left(\mathrm{B}u^n_{i,\lambda}(s)-\mathrm{B}u_{i,\lambda}(s)\right)}_{\mathbb{X}} \mathrm{d}s+\int_{s_{i}}^{t}\norm{\mathcal{S}(t, s)f(s,\tilde{x^n_s})-f(s,\tilde{x}_{s})}_{\mathbb{X}}\mathrm{d}s\\
		\nonumber&\le M\norm{\rho_{i}(s_{i}, \tilde{x^n}(t_{i}^{-}))-\rho_{i}(s_{i}, \tilde{x}(t_{i}^{-}))}_{\mathbb{X}} + \tilde{M}\norm{\left(\rho'_i(s_{i}, \tilde{x^n}(t_{i}^{-}))-\rho'_i(s_{i}, \tilde{x}(t_{i}^{-}))\right)}_{\mathbb{X}}\\\nonumber&\quad+\tilde{M}M_B\int_{s_{i}}^{t}\norm{u^n_{i,\lambda}(s)-u_{i,\lambda}(s)}_{\mathbb{X}} \mathrm{d}s
		+\tilde{M}\int_{s_{i}}^{t}\norm{f(s,\tilde{x^n}_{s})-f(s,\tilde{x}_{s})}_{\mathbb{X}}\mathrm{d}s\\
		&\rightarrow 0\text{ as }n\rightarrow\infty.
		\end{align*}
		Moreover, for $t\in (t_i,s_i],\ i=1,\ldots, N,$ applying the Assumption \ref{as} \textit{(H3)}, we get
		\begin{align*}
		\nonumber\norm{(\Phi_\lambda x^n)(t)-(\Phi_\lambda x)(t)}_{\mathbb{X}}&=\norm{\rho_{i}(t, \tilde{x^n}(t_{i}^{-}))-\rho_{i}(t, \tilde{x}(t_{i}^{-}))}_{\mathbb{X}}\\&\rightarrow 0 \ \text{ as }\ n\rightarrow \infty.
		\end{align*}
		Hence, it follows that $\Phi_\lambda$ is continuous.
		
		\vskip 0.1in 
		\noindent\textbf{Step (3): } \emph{$\Phi_\lambda$ is a compact operator}. To achieve this goal, we use the infinite-dimensional version of the Ascoli-Arzela theorem (see, Theorem 3.7, Chapter 2, \cite{XJ}). To apply the Ascoli-Arzela theorem, in view of  Step 1, it is enough to show that 
		\begin{itemize}
			\item [(i)] the image of $\mathcal{E}_r$ under $\Phi_{\lambda}$ is equicontinuous,
			\item [(ii)] for an arbitrary $t\in J$, the set $\mathcal{V}(t)=\{(\Phi_\lambda x)(t):x\in \mathcal{E}_r\}$ is relatively compact.
		\end{itemize}
		
		Firstly, we claim that the image of $\mathcal{E}_r$ under $\Phi_\lambda$ is equicontinuous. For $\tau_1,\tau_2\in[0,t_1]$ with  $\tau_1<\tau_2$ and $x \in \mathcal{E}_r$, we compute
		the following
		\begin{align}\label{4.35}
		\nonumber &\norm{(\Phi_\lambda x)(\tau_2)-(\Phi_\lambda x)(\tau_1)}_{\mathbb{X}}\\\nonumber&\le \norm{\mathcal{C}(\tau_2, 0)\phi(0)-\mathcal{C}(\tau_1, 0)\phi(0)}_{\mathbb{X}} +\norm{\mathcal{C}(\tau_2, 0)g(\tilde{x})-\mathcal{C}(\tau_1, 0)g(\tilde{x})}_{\mathbb{X}}\\\nonumber&\quad +\norm{[\mathcal{S}(\tau_2,0)-\mathcal{S}(\tau_1,0)]\eta}_{\mathbb{X}}+\norm{[\mathcal{S}(\tau_2,0)-\mathcal{S}(\tau_1,0)]h(\tilde{x})}_{\mathbb{X}}\\\nonumber&\quad + \int_{\tau_1}^{\tau_2}\norm{\mathcal{S}(\tau_2,s)f(s,\tilde{x}_{s})}_{\mathbb{X}}\mathrm{d}s+ \int_{\tau_1}^{\tau_2}\norm{\mathcal{S}(\tau_2,s)\mathrm{B}u_{0,\lambda}(s)}_{\mathbb{X}}\mathrm{d}s\\\nonumber&\quad+\int_{0}^{\tau_1}\norm{\mathcal{S}(\tau_2,s)-\mathcal{S}(\tau_1,s)\mathrm{B}u_{0,\lambda}(s)}_{\mathbb{X}}\mathrm{d}s\\\nonumber&\le\norm{\mathcal{C}(\tau_2, 0)\phi(0)-\mathcal{C}(\tau_1, 0)\phi(0)}_{\mathbb{X}}+\norm{\mathcal{C}(\tau_2, 0)h(\tilde{x})-\mathcal{C}(\tau_1,0)h(\tilde{x})}_{\mathbb{X}}\\ \nonumber&\quad
		+\norm{\mathcal{S}(\tau_2,0)-\mathcal{S}(\tau_1,0)}_{\mathcal{L}(\mathbb{X})}\norm{\eta}_{\mathbb{X}}+\norm{\mathcal{S}(\tau_2,0)-\mathcal{S}(\tau_1,0)}_{\mathcal{L}(\mathbb{X})}(r_1+1)\\\nonumber&\quad+\tilde{M}\int_{\tau_1}^{\tau_2}\gamma(s)\mathrm{d}s+\frac{(\tilde{M}M_B)^2}{\lambda}\left[\tilde{K}+K(r_1+1)+\tilde{M}\norm{\gamma}_{\mathrm{L}^1(J_1;\mathbb{X})}\right](\tau_2-\tau_1)\\&\quad+\!\!\sup\limits_{t\in[0,\tau_1]}\norm{\mathcal{S}(\tau_2,t)-\mathcal{S}(\tau_1,t)}_{\mathcal{L}(\mathbb{X})}M_B\int_{0}^{\tau_1}\norm{u_{0,\lambda}(s)}_{\mathbb{U}}\mathrm{d}s\\\nonumber&\quad+\sup\limits_{t\in[0,\tau_1]}\norm{\mathcal{S}(\tau_2,t)-\mathcal{S}(\tau_1,t)}_{\mathbb{X}}\int_{0}^{\tau_1}\gamma(s)\mathrm{d}s.
		\end{align}
		Similarly for $\tau_1,\tau_2\in (s_i,t_{i+1}],\ i=1,\ldots,N\ \text{with}\ \tau_1<\tau_2 \text{ and }x\in\mathcal{E}_r,$ we estimate
		\begin{align}\label{4.36}
		\nonumber &\norm{(\Phi_\lambda x)(\tau_2)-(\Phi_\lambda x)(\tau_1)}_{\mathbb{X}}\\\nonumber&\le\norm{\mathcal{C}(\tau_2, s_{i})\rho_{i}(s_{i}, \tilde{x}(t_{i}^{-}))-\mathcal{C}(\tau_1, s_{i})\rho_{i}(s_{i}, \tilde{x}(t_{i}^{-}))}_{\mathbb{X}} + \norm{\left(\mathcal{S}(\tau_2, s_{i})-\mathcal{S}(\tau_1, s_{i})\right)\rho'_i(s_{i}, \tilde{x}(t_{i}^{-}))}_{\mathbb{X}}\\\nonumber&\quad+\int_{\tau_1}^{\tau_2}\norm{\mathcal{S}(\tau_2, s)\mathrm{B}u_{i,\lambda}(s)}_{\mathbb{X}} \mathrm{d}s+\int_{\tau_1}^{\tau_2}\norm{\mathcal{S}(\tau_2, s)f(s,\tilde{x}_{s})}_{\mathbb{X}}\mathrm{d}s\\\nonumber &\quad+\int_{s_i}^{\tau_1}\norm{\left(\mathcal{S}(\tau_2, s)-\mathcal{S}(\tau_2, s)\right)\mathrm{B}u_{i,\lambda}(s)}_{\mathbb{X}} \mathrm{d}s+\int_{s_{i}}^{\tau_1}\norm{\left(\mathcal{S}(\tau_2, s)-\mathcal{S}(\tau_1, s)\right)f(s,\tilde{x}_{s})}_{\mathbb{X}}\mathrm{d}s\\
		\nonumber&\le\norm{\mathcal{C}(\tau_2, s_{i})\rho_{i}(s_{i}, \tilde{x}(t_{i}^{-}))-\mathcal{C}(\tau_1, s_{i})\rho_{i}(s_{i}, \tilde{x}(t_{i}^{-}))}_{\mathbb{X}}\!\!+\!  \norm{\mathcal{S}(\tau_2, s_{i})-\mathcal{S}(\tau_1, s_{i})}_{\mathcal{L}(\mathbb{X})}e_i\\\nonumber&\quad +\frac{(\tilde{M}M_B)^2}{\lambda}\left[K_i+Kr_1+\tilde{M}\norm{\gamma}_{\mathrm{L}^1(J;\mathbb{R}^+)}\right](\tau_2-\tau_1)+\tilde{M}\int_{\tau_1}^{\tau_2}\gamma(s)\mathrm{d}s\\&\quad+\sup\limits_{t\in[s_i,\tau_1]}\norm{\mathcal{S}(\tau_2, t)-\mathcal{S}(\tau_1, t)}_{\mathcal{L}(\mathbb{X})}M_B\int_{s_{i}}^{\tau_1}\norm{u_{i,\lambda}(s)}_{\mathbb{U}}\mathrm{d}s\nonumber\\&\quad+\sup\limits_{t\in[s_i,\tau_1]}\norm{\mathcal{S}(\tau_2, t)-\mathcal{S}(\tau_1, t)}_{\mathcal{L}(\mathbb{X})} \int_{s_{i}}^{\tau_1}\gamma(s)\mathrm{d}s
		\end{align}
		Moreover, for $\tau_{1},\tau_2\in (t_i,s_i],\text{ with }\tau_1<\tau_2\text{ and }x\in \mathcal{E}_r$, we have
		\begin{align} \label{4.37}
		\norm{(\Phi_\lambda x)(\tau_2)-(\Phi_\lambda x)(\tau_1)}_{\mathbb{X}} &=\norm{\rho_{i}(\tau_2, \tilde{x}(t_{i}^{-}))-\rho_{i}(\tau_1, \tilde{x}(t_{i}^{-}))}_{\mathbb{X}}.
		\end{align}
		Using the facts, the operator $\mathcal{C}(t,s)x,\ x\in\mathbb{X}$ is uniformly continuous for $t,s\in J$, the operator $\mathcal{S}(\cdot,s)$ is Lipschitz continuous for all $s\in J$ in uniform operator topology (see \eqref{25}) and the continuity of the impulses $\rho_{i}(\cdot, x)$ for each $x\in\mathbb{X}$, we obtain that the right hand side of the expressions \eqref{4.35}, \eqref{4.36} and \eqref{4.37}, converge to zero as $|\tau_2-\tau_1|\rightarrow 0$. Hence, the image of $\mathcal{E}_r$ under $\Phi_\lambda$ is equicontinuous.

		Next, we verify that for each $\lambda>0,$ the image of $\mathcal{E}_r$ under the operator $\Phi_\lambda$ is relatively compact. For this, we show that the set $\mathcal{V} (t) = \{(\Phi_{\lambda}x)(t) : x \in \mathcal{E}_r\}$ is relatively compact for every $t \in J$. The set $\mathcal{V}(t),$ for $t\in J$, is relatively compact in $\mathcal{E}_r$ follows by the facts that the  operator $\mathcal{S}(t,s)$ is compact for $t\le s$, the impulses $\rho_{i}(t,\cdot), i=1,\ldots,N,\ t\in J_1$ and the functions $g(\cdot)$ are completely continuous for each $x\in\mathrm{PC}(J_q;\mathbb{X})$ and also the compactness of the operator $(\mathcal{Q})(\cdot) =\int_{0}^{\cdot}\mathcal{S}(\cdot,s)f(s)\mathrm{d}s$ (Lemma 3.2, Corollary 3.3, Chapter 3, \cite{XJ}). Therefore, the set $\mathcal{V} (t) = \{(\Phi_{\lambda}x)(t) : x \in \mathcal{E}_r\},$ for each $t \in J$ is relatively compact in $\mathcal{E}_r$.
		
		Hence, the operator $\Phi_\lambda$ is compact in view of Ascoli-Arzela theorem. Then by an application of \emph{Schauder's fixed point theorem} yields that the operator $\Phi_\lambda$ has a fixed point $\tilde{x}(\cdot)$ in $\mathcal{E}_r$. Thus, by the definition of $\tilde{x}(\cdot)$, we obtain that  $\tilde{x}(\cdot)$ is a mild solution of the system \eqref{1.1}-\eqref{1.5}.
	\end{proof}
	In next theorem, we determine the approximate controllability of the system \eqref{1.1}-\eqref{1.5}. For this, we impose the following assumption on $f(\cdot,\cdot)$.
	\begin{Ass}\label{as1}
		\begin{itemize}
			\item[\textit{(H5)}] The function $f: J_1\times \mathcal{D}\rightarrow\mathbb{X}$ satisfies the Assumption (\textit{H2})(i) and  there exists a constant $\mathcal{N}(t)\in\mathrm{L}^2(J_1;\mathbb{R}^+)$ such that
			$$ \norm{f(t,\phi)}_{\mathbb{X}}\le \mathcal{N}(t), \text{ for a.e.}\ t\in J_1 \text{ and }\phi\in \mathcal{D}.$$
		\end{itemize}
	\end{Ass}
	\begin{thm}
		Suppose that the Assumptions (H0)-(H1), (H3)-(H5) and the condition \eqref{4.18} hold. Then the system \eqref{1.1}-\eqref{1.5} is approximately controllable.
	\end{thm}
	\begin{proof}
		By using Theorem \ref{th3.1}, for every $\lambda>0\text{ and }x_T\in \mathbb{X},$ there exists a mild solution $x^\lambda(\cdot)$  with the control defined in \eqref{4.1} satisfying
		\begin{equation}\label{mild}
		x^\lambda(t)=\left\{
		\begin{aligned}
		& \mathcal{C}(t, 0)\phi(0) + \mathcal{S}(t,0)\eta +\int_{0}^{t}\mathcal{S}(t, s)[\mathrm{B}u_{\lambda}(s)+f(s,\tilde{x^\lambda}_{s})] \mathrm{d}s, \  t \in  (0,t_{1} ], \\
		&\rho_{i}(t, \tilde{x^\lambda}(t_{i}^{-})) ,\ t \in  	\bigcup_{i=1}^{N} (t_{i},s_{i} ],\\
		& \mathcal{C}(t, s_{i})\rho_{i}(s_{i}, \tilde{x^\lambda}(t_{i}^{-})) + \mathcal{S}(t, s_{i})\rho'_i(s_{i}, \tilde{x^\lambda}(t_{i}^{-})) \\ &\quad +\int_{s_{i}}^{t}\mathcal{S}(t, s)[\mathrm{B}u_{\lambda}(s)+f(s,\tilde{x^\lambda}_{s})] \mathrm{d}s, \ t \in  	\bigcup_{i=1}^{N} (s_{i},t_{i+1} ].
		\end{aligned}
		\right.
		\end{equation}  
		Next, we estimate 
		\begin{align}\label{3.43}
		\nonumber x^\lambda(T)&=\mathcal{C}(T, s_{N})\rho_{N}(s_{N}, \tilde{x^\lambda}(t_{N}^{-})) + \mathcal{S}(T, s_{N})\rho'_{N}(s_{N}, \tilde{x^\lambda}(t_{N}^{-}))\\\nonumber&\quad +\int_{s_{N}}^{T}\mathcal{S}(T, s)[\mathrm{B}u_{N,\lambda}(s)+f(s,\tilde{x^\lambda}_{s})] \mathrm{d}s\\
		\nonumber&=\mathcal{C}(T, s_{N})\rho_{N}(s_{N}, \tilde{x^\lambda}(t_{N}^{-})) + \mathcal{S}(T, s_{N})\rho'_{N}(s_{N}, \tilde{x^\lambda}(t_{N}^{-}))+\int_{s_{N}}^{T}\mathcal{S}(T, s)f(s,\tilde{x^\lambda}_{s}) \mathrm{d}s\\ \nonumber&\quad+\int_{s_{N}}^{T}\mathcal{S}(T,s)\mathrm{B}\mathrm{B}^*S^*(T,s)\mathcal{J}\left[\mathrm{R}(\lambda,\Psi_{s_N}^{T})g_N(x^\lambda(\cdot))\right]\mathrm{d}s\\\nonumber
		&=\mathcal{C}(T, s_{N})\rho_{N}(s_{N}, \tilde{x^\lambda}(t_{N}^{-})) + \mathcal{S}(T, s_{N})\rho'_{N}(s_{N}, \tilde{x^\lambda}(t_{N}^{-}))+\int_{s_{N}}^{T}\mathcal{S}(T, s)f(s,\tilde{x^\lambda}_{s}) \mathrm{d}s\\\nonumber&\quad+\Psi_{s_N}^{T}\mathcal{J}\left[\mathrm{R}(\lambda,\Psi_{s_N}^{T})g_N(x^\lambda(\cdot))\right]\\
		&=x_T-\lambda\mathrm{R}(\lambda,\Psi_{s_N}^{T})g_N(x^\lambda(\cdot)).
		\end{align}
		Moreover, since $ x^{\lambda}\in\mathcal{E}_r $ implies that the sequence $ \{x^{\lambda}(t)\}_{\lambda>0},$  for each $t\in J$ is bounded in $\mathbb{X}$. Then by an application of the Banach-Alaoglu theorem, we can find a subsequence, still denoted as $ x^{\lambda},$ such that 
		\begin{align*}
		x^\lambda(t)\xrightharpoonup{w}z(t) \ \mbox{ in }\ \mathbb{X} \ \ \mbox{as}\  \ \lambda\to0^+,\ t\in J.
		\end{align*}
		Using the condition (\textit{$H3$}) of Assumption \ref{as}, we obtain
		\begin{align}
		\label{3.40}	\rho_N(t,x^\lambda(t_N^-))&\to \rho_N(t,z(t_N^-)) \ \mbox{ in }\ \mathbb{X} \ \ \mbox{as}\  \ \lambda\to0^+,\\
		\label{3.41}	\rho'_N(t,x^\lambda(t_N^-))&\xrightharpoonup{{w}}\rho'_N(t,z(t_N^-)) \ \mbox{ in }\ \mathbb{X} \ \ \mbox{as}\  \ \lambda\to0^+.
		\end{align}
		Furthermore,  by condition \textit{(H5)} of Assumption \ref{as1}, we have 
		\begin{align}
		\nonumber	\int_{s_N}^{T}\norm{f(t,\tilde{x^\lambda_{s}})}^2_{\mathbb{X}} \mathrm{d}s\le \int_{s_N}^{T}\mathcal{N}^2(s)\mathrm{d}s<+\infty.
		\end{align}
		Therefore, the sequence $\{f(\cdot,\tilde{x^\lambda_{s}}):\lambda>0\}$ in $\mathrm{L}^2([s_N,T];\mathbb{X})$ is bounded. Once again by the Banach-Alaoglu theorem, there exists a subsequence still denoted by $\{f(\cdot,\tilde{x^\lambda_{s}}):\lambda>0\}$ such that
		\begin{align}\label{3.42}
		f(\cdot,\tilde{x^\lambda_{s}})\xrightharpoonup{{w}} f(\cdot) \ \text{ in } \ \mathrm{L}^2([s_N,T];\mathbb{X}).
		\end{align}
		Next, we evaluate
		\begin{align}\label{3.44}
		\nonumber \norm{g_N(x^\lambda(\cdot))-\omega}_{\mathbb{X}}&\le \norm{\mathcal{C}(T,s_N)[\rho_{N}(s_{N},\tilde{x^\lambda}(t_{N}^{-}))-\rho_{N}(s_{N},z(t_{N}^{-}))]}_{\mathbb{X}}\\\nonumber&\quad+\norm{\mathcal{S}(T,s_N)[\rho'_{N}(s_{N},\tilde{x^\lambda}(t_{N}^{-}))-\rho'_{N}(s_{N},z(t_{N}^{-}))]}_{\mathbb{X}}\\\nonumber&\quad+\int_{s_{N}}^{T}\norm{\mathcal{S}(T, s)\left(f(s,\tilde{x^\lambda_s})-f(s)\right)}_{\mathbb{X}} \mathrm{d}s\\\nonumber&\le M\norm{\rho_{N}(s_{N},\tilde{x^\lambda}(t_{N}^{-}))-\rho_{N}(s_{N},z(t_{N}^{-}))}_{\mathbb{X}}\\\nonumber&\quad+\norm{\mathcal{S}(T,s_N)[\rho'_{N}(s_{N},\tilde{x^\lambda}(t_{N}^{-}))-\rho'_{N}(s_{N},z(t_{N}^{-}))]}_{\mathbb{X}}\\\nonumber&\quad+\int_{s_{N}}^{T}\norm{\mathcal{S}(T, s)\left(f(s,\tilde{x^\lambda}_{s})-f(s)\right)}_{\mathbb{X}} \mathrm{d}s\\&\rightarrow 0 \text{ as }\lambda\rightarrow 0^+.
		\end{align}
		where
		\begin{align}
		\nonumber \omega&=x_{T}-\mathcal{C}(T, s_N)\rho_{N}(s_{N}, z(t_{N}^{-}))  - \mathcal{S}(T,s_N)\rho'_{N}(s_{N}, z(t_{N}^{-}))  -\int_{s_N}^{T}\mathcal{S}(T, s)f({s})\mathrm{d}s,
		\end{align}
		and  we used the convergences \eqref{3.40},\eqref{3.41},\eqref{3.42} and the compactness of the operators $\mathcal{S}(\cdot,\cdot):\mathbb{X}\to\mathbb{X}$ and $(\mathcal{Q})(\cdot) =\int_{0}^{\cdot}\mathcal{S}(\cdot,s)f(s)\mathrm{d}s:\mathrm{L}^2(J;\mathbb{X}) \rightarrow \mathcal{C}(J;\mathbb{X})$ (Lemma 3.2, Corollary 3.3, Chapter 3, \cite{XJ}).
		
		Finally, by using the expression \eqref{3.43}, the convergence \eqref{3.44} and Assumption \ref{as} \textit{(H0)}, we obtain
		\begin{align}\nonumber\norm{x^\lambda(T)-x_T}_{\mathbb{X}}&= \norm{\lambda\mathrm{R}(\lambda,\Psi_{s_N}^{T})g_N(x^\lambda(\cdot))}_{\mathbb{X}}\\
		&\le \nonumber \norm{\lambda\mathrm{R}(\lambda,\Psi_{s_N}^{T})g_N(x^\lambda(\cdot))-\omega}_{\mathbb{X}}+\norm{\lambda\mathrm{R}(\lambda,\Psi_{s_N}^{T})\omega}_{\mathbb{X}}\\
		\nonumber &\le\norm{\lambda\mathrm{R}(\lambda,\Psi_{s_N}^{T})}_{\mathcal{L}(\mathbb{X})}\norm{g_N(x^\lambda(\cdot))-\omega}_{\mathbb{X}}+\norm{\lambda\mathrm{R}(\lambda,\Psi_{s_N}^{T})\omega}_{\mathbb{X}}\\
		&\rightarrow 0 \ \text{ as }\ \lambda\rightarrow 0^+,
		\end{align}
	 Hence, the system \eqref{1.1}-\eqref{1.5} is approximately controllable on $J$.
	\end{proof}

	\section{Application}\label{Appl} \setcounter{equation}{0} 
	In this section, we investigate the approximate controllability of the non-autonomous wave equation with non-instantaneous impulses and finite delay. 
	\begin{Ex}\label{ex} Let us consider the following wave equation: 
		\begin{equation}\label{ex1}
		\left\{
		\begin{aligned}
		\frac{\partial^2y(t,\xi)}{\partial t^2}&= \frac{\partial^2y(t,\xi)}{\partial \xi^2}+b(t)\frac{\partial y(t,\xi)}{\partial \xi}+\mu(t,\xi)+k_0\cos\left(\frac{2\pi t}{T}\right)\sin (y(t-r,\xi)), \\  & \qquad t\in\cup_{i=0}^{N} (s_i, t_{i+1}]\subset J=[0,T], \ \xi\in[0,2\pi],\\
		y(t,\xi)&=\rho_i(t,y(t_i^-,\xi)),\ t\in(t_i,s_i], \ i=1,\ldots,N,\ \xi\in[0,2\pi],\\
		\frac{\partial y(t,\xi)}{\partial t}&=\frac{\partial \rho_i(t,y(t_i^-,\xi))}{\partial t},t\in(t_i,s_i], \ i=1,\ldots,N,\ \xi\in[0,2\pi],\\
		y(t,0)&=y(t,2\pi),\ t\in J,\\
		y(0,\xi)&=\varphi(0,\xi)+ \int_{-r}^{T}\varrho(s)\log(1+|y(s, \xi)|)\mathrm{d} s, \ \xi\in[0,2\pi], \\ \frac{\partial y(0,\xi)}{\partial t}&=\zeta_0(\xi)+\sum\limits_{j=1}^{q}c_{j}y(\tau_{j}, \xi),\  -r<\tau_{1}<\tau_{2},\ldots,<\tau_{q}<T, \ \xi\in[0,2\pi],\\
		y(\theta,\xi)&=\varphi(\theta,\xi), \ \xi\in[0,2\pi], \ \theta\in[-r, 0),
		\end{aligned}
		\right.
		\end{equation}
		where $\varphi: [-r,0]\times [0,2\pi]\rightarrow \mathbb{R}$ is a piecewise continuous function, the function $\mu:J\times[0,2\pi]\to\mathbb{C}$ is continuous in $t$ and satisfies $\mu(t,0)=\mu(t,2\pi),$ for each $t\in J,$ and the function  $b:J\to\mathbb{R}$ is continuous. The functions $\varrho$ and $\rho_i$ for $i=1,\ldots,N$ satisfy suitable conditions, which will be described later.
	\end{Ex}
	Let $\mathbb{X}_p=\mathrm{L}^{p}(\mathbb{T};\mathbb{C}),$ for $p\in[2,\infty),$ be the space of $p$-integrable  functions (Lebesgue) from $\mathbb{R}$ into $\mathbb{C}$ with period $2\pi$, where $\mathbb{T}$ is the quotient group $\mathbb{R}/2\pi\mathbb{Z}$ and $\mathbb{U}=\mathrm{L}^{2}(\mathbb{T};\mathbb{C})$. We consider the operator $\mathrm{A}_pz(\xi)=z''(\xi)$ with the domain $\mathrm{D}(\mathrm{A}_p)= \mathrm{W}^{2,p}(\mathbb{T};\mathbb{C}).$ The operator $\mathrm{A}_p$ can be written as
	\begin{align*}
	\mathrm{A}_pz&= \sum_{n=1}^{\infty}-n^{2}\langle z, w_{n} \rangle  w_{n},\ \langle z,w_n\rangle :=\int_0^{2\pi}z(\xi)w_n(\xi)\mathrm{d}\xi,
	\end{align*}
	where $-n^2$($n\in\mathbb{Z}$) and $w_n(\xi)=\frac{1}{2\pi}e^{in\xi}$, are the eigenvalues and the corresponding normalized eigenfunctions of the operator $\mathrm{A}_p$. Moreover, we can verify in a similar way as in \cite{SS} that the operator $\mathrm{A}_p$ generate a cosine family $\mathrm{C}_{0,p}(t), t\in\mathbb{R}$ on $\mathbb{X}_p$ which is strongly continuous and the associated sine family $\mathrm{S}_{0,p}(t),t\in\mathbb{R}$ on $\mathbb{X}_p$ is compact. Thus, the condition \textit{(H1)} of Assumption \ref{as} holds. Furthermore, the cosine and sine families can explicitly be written as
	\begin{align*}
	\mathrm{C}_{0,p}(t)z&=\sum_{n\in\mathbb{Z}}\cos(nt)\langle z, w_{n} \rangle  w_{n},\ z\in\mathbb{X}_p,\nonumber\\
	\mathrm{S}_{0,p}(t)z&=t\langle z, w_0 \rangle w_0+\sum_{n\in\mathbb{Z},n\neq0}\frac{1}{n}\sin(nt)\langle z, w_{n} \rangle  w_{n},\ z\in\mathbb{X}_p.
	\end{align*} 
	Let us define $$\mathrm{F}_p(t)z(\xi)=b(t)z'(\xi)\ \text{ on }\ \mathrm{W}^{1,p}(\mathbb{T};\mathbb{C}).$$ It is easy to verify that the linear operator $\mathrm{A}_p(t)=\mathrm{A}_p+\mathrm{F}_p(t)$  is closed. Next, we show that $\mathrm{A}_p+\mathrm{F}_p(t)$ generates an evolution family. For this, we consider the following scalar initial value problem:
	\begin{equation}\label{IVP}
	\left\{
	\begin{aligned}
	h''(t)&=-n^2h(t)+inb(t)h(t),\\
	h(s)&=0, \ h'(s)=w.
	\end{aligned}
	\right.
	\end{equation}
	The solution of the above equation satisfies the integral equation 
	\begin{align*}
	h(t,s)=\frac{w}{n}\sin [n(t-s)]+i\int_{s}^{t}\sin [n(t-\tau)]b(\tau)h(\tau)\mathrm{d}\tau.
	\end{align*}
	By using the Gronwall-Bellman lemma, we obtain 
	\begin{align}\label{4.3}
	|h(t,s)|\leq\frac{|w|}{n}e^{\delta({t-s})},\ \mbox{for}\ s\le t,
	\end{align}
	where $\delta=\sup_{t\in J}|b(t)|$. Let us denote the solution of the initial value problem \eqref{IVP} by $h_n(t,s),$ for each $n\in\mathbb{Z}$ and define 
	\begin{align}
	\mathcal{S}_p(t,s)z=\sum_{n\in\mathbb{Z}}h_n(t,s)\langle z, w_n \rangle w_n, \ z\in\mathbb{X}_p.
	\end{align}  
	The estimate \eqref{4.3} ensures that $\mathcal{S}_p(t,s):\mathbb{X}_p\to\mathbb{X}_p$ is well defined and it satisfies the conditions of Definition \ref{def2.1}.
	
	Let us define $$x(t)(\xi):=y(t,\xi),\ \text{ for }\ t\in J\ \text{ and }\ \xi\in[0,2\pi],$$ and the linear operator $\mathrm{B}:\mathbb{U}\to\mathbb{X}_p$ as  $$\mathrm{B}u(t)(\xi):=\mu(t,\xi)=\int_{0}^{2\pi}K(\zeta,\xi)u(t)(\zeta)\mathrm{d}\zeta=\ t\in J,\ \xi\in [0,2\pi],$$ where $K\in\mathrm{C}([0,2\pi]\times[0,2\pi];\mathbb{R})$ is an one-one operator with $K(\zeta,\xi)=K(\xi,\zeta)$. 
	Let us estimate 
	\begin{align*}
	\norm{\mathrm{B}u(t)}_{\mathbb{X}_p}^p=\int_{0}^{2\pi}\left|\int_{0}^{2\pi}K(\zeta,\xi)u(t)(\zeta)\mathrm{d}\zeta\right|^p\mathrm{d}\xi.
	\end{align*}
	Applying the Cauchy-Schwarz inequality, we have
	\begin{align*}
	\norm{\mathrm{B}u(t)}_{\mathbb{X}_p}^p&\le\int_{0}^{2\pi}\left[\left(\int_{0}^{2\pi}|K(\zeta,\xi)|^2\mathrm{d}\zeta\right)^{\frac{1}{2}}\left(\int_{0}^{2\pi}|u(t)(\zeta)|^2\mathrm{d}\zeta\right)^{\frac{1}{2}}\right]^{p}\mathrm{d}\xi\\&=\left(\int_{0}^{2\pi}|u(t)(\zeta)|^2\mathrm{d}\zeta\right)^{\frac{p}{2}}\int_{0}^{2\pi}\left(\int_{0}^{2\pi}|K(\zeta,\xi)|^2\mathrm{d}\zeta\right)^{\frac{p}{2}}\mathrm{d}\xi.
	\end{align*}
	Since the kernel $K(\cdot,\cdot)$ is continuous, we arrive at
	\begin{align*}
	\norm{\mathrm{B}u(t)}_{\mathbb{X}_p}\le C\norm{u(t)}_{\mathbb{U}},
	\end{align*}
	so that we get 
	$	\norm{\mathrm{B}}_{\mathcal{L}(\mathbb{U};\mathbb{X}_p)}\le C.$
	Hence, the operator $\mathrm{B}$ is bounded. Moreover, the symmetry of the kernel  implies that the operator $\mathrm{B}=\mathrm{B}^*$ (self adjoint). In particular, one can take $K(\xi,\zeta)=1+\xi^2+\zeta^2,\ \mbox{for all}\ \xi, \zeta\in [0,2\pi]$.
	
	Let us now define $f:J_1=\bigcup_{i=0}^{N}[s_i,t_{i+1}]\times \mathcal{D}\rightarrow  \mathbb{X}_p$  as  
	\begin{align}
	f(t,x_{t})(\xi):= k_{0}\cos\left(\frac{2\pi t}{T}\right)\sin(x_{t}), \ \xi\in[0,2\pi],
	\end{align}
	where $k_{0}$ is some positive constant and $\mathcal{D}$ is given in \eqref{1.2}. Clearly, $f$ is continuous  and 
	\begin{align*}
	\left\|f(t,x_t)\right\|_{\mathrm{L}^p}=k_0\left(\int_{0}^{2\pi}\left|\cos\left(\frac{2\pi t}{T}\right)\sin(x_{t}(\zeta))\right|^p\mathrm{d}\zeta\right)^{1/p}\le k_0({2\pi})^{\frac{1}{p}} 
	\left|\cos\left(2\pi t\right)\right|=\gamma(t).
	\end{align*}
	It is clear that the function $\gamma(t)=k_0({2\pi})^{\frac{1}{p}}\left|\cos\left(\frac{2\pi t}{T}\right)\right|\in\mathrm{L^{1}}(J_1;\mathbb{R^{+}})\cap\mathrm{L^{2}}(J_1;\mathbb{R^{+}})$
	Hence, the condition (\textit{$H2$}) of Assumption \ref{as} and the condition (\textit{$H5$}) of Assumption \ref{as1} is satisfied. 
	
	Next, we define the impulse functions $\rho_{i,p}:[t_i,s_i]\times\mathbb{X}_p\to\mathbb{X}_p,$ for each $i=1,\ldots,N$ are defined as 
	\begin{align*}
	\rho_{i,p}(t,x)(\xi):=\int_{0}^{2\pi}g_i(t,\xi,z)\cos^2(x(t_i^-)z)\mathrm{d}z, 
	\end{align*}
	where, $g_i\in\mathrm{C}^{1}([0,T]\times[0,2\pi]^2;\mathbb{R})$. It is not difficult to verify that the impulses $\rho_{i,p},$ for $i=1,\ldots,N$ satisfy the condition \textit{$(H3)$} of Assumption \ref{as}. 
	
	Finally, we consider the nonlocal initial conditions $ g_p,h_p :\mathrm{PC}([-r,T];\mathbb{X}_p)\rightarrow \mathbb{X}_p$ as
	\begin{align*}
	g_p(x)(t)&=g(y(t, \xi)),\ \xi\in[0,2\pi],\\
	h_p(x)(t)&=h(y(t, \xi)),\ \xi\in[0,2\pi].
	\end{align*}
	Let us choose 
	\begin{align*}
	g_p(y(t, \xi))&= \int_{-r}^{\tau}\varrho(s)\log(1+|y(s, \xi)|)d s,\ \xi\in[0, 2\pi], \\
	h_p(y(t, \xi))&= \sum\limits_{j=1}^{q}c_{j}y(\tau_{j}, \xi), \  -r<\tau_{1}<\tau_{2},\ldots,<\tau_{q}<T,\ \xi\in[0, 2\pi],
	\end{align*}
	where $ \varrho\in\mathrm{L}^{1}([-r,\tau]; \mathbb{R})$ and $c _{j}$ for $j=1, \dots, q,$ are positive constants. Let $\left\|\varrho\right\|_{\mathrm{L}^1([-r,\tau];\mathbb{R})}\le l$. Hence, the functions $g_p(\cdot)$ and $h_p(\cdot)$ satisfy the condition \textit{(H4)} of Assumption \ref{as}, whenever the constant $l$ and  $c_j$'s are small enough (see \cite{lia} for more details).
	
	Using the above notions, the system \eqref{ex} can be expressed as \eqref{1.1}-\eqref{1.5} and it satisfies Assumption \ref{as} \textit{(H1)-(H4)} and Assumption \ref{as1} (\textit{H5}). Finally, we verify that the linear problem corresponding to the system \eqref{1.1}-\eqref{1.5} is approximately controllable. To complete this, we consider (see Remark \ref{rm})
	$$\mathrm{B}^*\mathcal{S}_{p}(T,t)^*x^*=0,\ \mbox{ for any}\ x^*\in\mathbb{X}^*,\ t\in J.$$  By using definition of $\mathrm{B}^*$, we obtain 
	\begin{align*}
 \int_0^{2\pi}K(\zeta,\xi)	\mathcal{S}_{p}(T,t)^*x^*(\zeta)\mathrm{d}\zeta=0, \ \mbox{ for any}\ x^*\in\mathbb{X}^*, t\in J.
	\end{align*}
Since $K(\cdot,\cdot)$ is one-one, then we have 
	\begin{align*}
	\mathcal{S}_{p}(T,t)^*x^*=0 \ \mbox{ for all }\ t\in J.
	\end{align*}
	The property of $\mathcal{S}_{p}(T,t)^*$ also implies that 
	\begin{align*}
	\frac{\partial}{\partial t}\mathcal{S}_{p}(T,t)^*x^*=0  \ \mbox{ for all }\ t\in J.
	\end{align*}
	Using the condition (D1)-(b) in Definition \ref{def2.1}, we obtain $x^*=0$. Thus, the linear system is approximately controllable. Finally, Theorem \ref{th3.1} yield that the system \eqref{ex} is approximately controllable.
	
	\medskip\noindent
	{\bf Acknowledgments:} S. Arora would like to thank Council of Scientific and Industrial Research, New Delhi, Government of India (File No. 09/143(0931)/2013 EMR-I), for financial support to carry out his research work and Department of Mathematics, Indian Institute of Technology Roorkee (IIT Roorkee), for providing stimulating scientific environment and resources.


\begin{thebibliography}{10}
		
		
		\bibitem{AH} H.M. Ahmed,  M.M. El-Borai,  A.O. El Bab and M.E. Ramadan, Approximate controllability of non-instantaneous impulsive Hilfer fractional integrodifferential equations with fractional Brownian motion, {\em Bound. Value Probl.}, {\bf 2020} (2020), 1-25.
		
		
		
		
		
		\bibitem{AM} S. Arora, M. T. Mohan and J. Dabas, Approximate controllability of the non-autonomous impulsive evolution equation with state-dependent delay in Banach space, {\em Nonlinear
			Anal. Hybrid System}, {\bf 39} (2012).
		
		
		\bibitem{AS} S. Arora, S. Singh, J. Dabas and M. T. Mohan, Approximate controllability of semilinear impulsive functional differential system with nonlocal conditions, {\em IMA J. Math. Control Inform.}, {\bf 37} (2020), 1070–1088.
		
		\bibitem{AJ} S. Arora, M. T. Mohan, and J. Dabas, Approximate controllability of a Sobolev type impulsive functional evolution system in Banach spaces, {\em Math. Control. Relat. Fields}, (2020), doi: 10.3934/mcrf.2020049.
		
	
	
	\bibitem{UA}U. Arora and N. Sukavanam, Approximate controllability of second order semilinear stochastic system with nonlocal conditions, \emph{Appl. Math. Comput.}, {\bf 258} (2015), 111-119.
		
		\bibitem{VB} \text{V. Barbu}, \emph{Analysis and Control of Nonlinear Infinite Dimensional Systems}, Academic Press, New York, 1993. 
		
		\bibitem{VB1}  V. Barbu, \emph{Controllability and Stabilization of Parabolic Equations}, Springer International Publishing AG, 2018. 
		
		\bibitem{ABE} A.E. Bashirov, N.I. Mahmudov, On concepts of controllability for deterministic and stochastic systems, {\em SIAM J. Control Optim.}. {\bf 37}  (1999), 1808–1821.
		
		
		\bibitem{BM} M. Benchohra, J. Henderson and S. Ntouyas, {\em Impulsive Differential Equations and Inclusions}, Hindawi Publishing Corporation, New York (2006).
		
		\bibitem{Bo} J. Bochenek, Existence of the fundamental solution of a second order evolution equation, \emph{Ann. Polon. Math.}, {\bf 66} (1997), 15–35.
		
		\bibitem{by} L. Byszewski and V. Lakshmikantham, Theorem about the existence and uniqueness of a solution of a nonlocal abstract cauchy problem in a Banach space, \emph{Applicable Analysis}, {\bf 40} (1991), 11-19.
		
		\bibitem{FC} F. Chen, D. Sun and J. Shi, Periodicity in a food-limited population model with toxicants and state dependent delays, {\em J. Math. Anal. Appl.}, {\bf 288} (2003), 136-146.
		
		\bibitem{GDJZ}	\newblock G. Da Prato and J. Zabczyk, \newblock \emph{Ergodicity for Infinite Dimensional Systems}, \newblock London Mathematical Society Lecture Notes, Cambridge University Press, 1996.
		
		\bibitem{FM} M. Feckan and J. Wang, A general class of impulsive evolution equations, {\em Topol. Methods Nonlinear Anal}, {\bf 46} (2015), 915-933.
		
		\bibitem{FUX} X. Fu, Approximate controllability of semilinear non-autonomous evolution systems with state-dependent delay, {\em Evol. Equ. Control Theory}, {\bf 6} (2017), 517–534.
		
		\bibitem{FUH} X. Fu and H. Rong, Approximate controllability of semilinear non-autonomous evolutionary systems with nonlocal conditions, {\em Autom. Remote Control}, {\bf77} (2016), 428–442.
		
		\bibitem{AG} A. Grudzka and K. Rykaczewski, On approximate controllability of functional impulsive evolution inclusions in a Hilbert space, {\em J. Optimiz Theory App.}, {\bf 166} (2015), 414-439.
		
		\bibitem{HR} H.R., Hen\'{r}iquez, Existence of solutions of non-autonomous second order functional differential equations with infinite delay, {\em Nonlinear Anal.}, {\bf 74} (2011), 3333-3352.
		
		\bibitem{Er} E. Hern\'{a}ndez M, H.R., Hen\'{r}iquez and M.A. McKibben,  Existence results for abstract impulsive second order neutral functional differential equations, {\em Nonlinear Anal.}, {\bf 70} (2009), 2736-2751.
		
		\bibitem{HE} E. Hern\'{a}ndez and D. O'Regan, On a new class of abstract impulsive differential equations, {\em Proc. Am. Math. Soc.}, {\bf 141} (2013), 1641–1649.
		
 		\bibitem{KJ} J. Kisy\'{n}ski, On cosine operator functions and one parameter group of operators, {\em Studia Math.}, {\bf49} (1972), 93-105.
		
		\bibitem{MK} M. Kozak, An abstract second order temporally inhomogeneous linear differential equation II, {\em Univ.lagel. Acta Math.}, {\bf 32} (1995), 263-274.
		
		\bibitem{KA} A. Kumar, M. Muslim,  and R. Sakthivel, Controllability of the second order nonlinear differential equations with non-instantaneous impulses, {\em J Dyn Control Syst}, {\bf 24} (2018), 325-342.
		
		\bibitem{KV} A. Kumar, R. K. Vats and A. Kumar, Approximate controllability of second order nonautonomous system with finite delay, {\em J. Dyn. Control Syst.}, {\bf26} (2020), 611-627.
		
		\bibitem{LVB} V. Lakshmikantham, D. D. Bainov and P.S. Simeonov, {\em Theory of Impulsive Differential Equations}, World Scientific, Singapore (1989).
		
		\bibitem{XJ} X. J. Li and J. M. Yong, {\em Optimal Control Theory for Infinite-Dimensional Systems}, Systems \& Control: Foundations \& Applications, Birkh\"{a}user Boston, Inc., Boston, MA, 1995.
		
		\bibitem{Yl} Y. Lin, Time-dependent perturbation theory for abstract evolution equations of second order, \emph{Studia Math.}, {\bf 130} (1998), 263–274.
		
		\bibitem{lia} J. Liang, J.H. Liu and T.J. Xiao, Nonlocal impulsive problems for nonlinear differential equations in Banach spaces, \emph{Math. Comput. Modelling}, \textbf{49} (2009), 798--804.
		
		\bibitem{LA} A. Lunardi, On the linear heat equation with fading memory, {\em SIAM J. Math. Anal.}, {\bf 21} (1990), 1213-1224.
		
		\bibitem{MN} N.I. Mahmudov, Approximate controllability of semilinear deterministic and stochastic evolution equations in abstract spaces, {\em SIAM J. Control Optim.}, {\bf 42} (2003), 1604-1622.
		
		\bibitem{MV} N.I. Mahmudov, V. Vijayakumar and R. Murugesu, Approximate controllability of second order evolution differential inclusions in Hilbert spaces, {\em Mediterr. J. Math.}, {\bf13} (2016), 3433–3454.
		
		\bibitem{MM} M. Malik, A. Kumar and M. Feckan, Existence, uniqueness, and stability of solutions to second order nonlinear differential equations with non-instantaneous impulses, {\em J. King Saud Univ. Sci.}, {\bf 30} (2018),  204-213.
		
		\bibitem{nt} S.K. Ntouyas, Nonlocal initial and boundary value problems: a survey, \emph{Handbook of differential equations: ordinary differential equations}, {\bf 2}, Elsevier, 2006, 461-557.
		
		\bibitem{NJ} J.W. Nunziato, On heat conduction in materials with memory, {\em Quart. Appl. Math.}, {\bf 29} (1971), 187-204.
		
		\bibitem{Eo} E. Obrecht, Evolution operators for higher order abstract parabolic equations, \emph{Czechoslovak Math.J.}, {\bf 36} (1986), 210–222.
		
		\bibitem{Oe} E. Obrecht, The Cauchy problem for time-dependent abstract parabolic equations of higher order, \emph{J. Math. Anal. Appl. }, {\bf 125} (1987), 508–530.
		
		\bibitem{PM} M. Pierri, H.R. Henr\'{i}quez and A. Prokopczyk, Global solutions for abstract differential equations with non-instantaneous impulses, {\em Mediterr. J. Math}, {\bf49} (2016), 1685-1708.
		
		\bibitem{RK} K. Ravikumar, M. T. Mohan and A. Anguraj, Approximate controllability of a nonautonomous evolution equation in Banach spaces, {\em Numer. Algebra Control Optim.}, (2020).
		
		\bibitem{RSA} R. Sakthivel, E.R. Anandhi, N.I. Mahmudov, Approximate controllability of second order systems with state-dependent delay, {\em Numer. Funct. Anal. Optim.}, {\bf 29} (2008),  1347–1362.
		
		\bibitem{SH} H. Serizawa and M. Watanabe, Time-dependent perturbation for cosine families in Banach spaces, {\em Houston J Math}, {\bf 2} (1986), 579–586.
		
		\bibitem{SS} S. Singh, S. Arora, M. T. Mohan and J. dabas, Approximate controllability  of second order impulsive systems with state-dependent delay in Banach spaces, {\em Evol. Equ. Control Theory}, (2020), doi: 10.3934/eect.2020103.
		
		\bibitem{SH1} L. Shu, X.B. Shu and  J. Mao, Approximate controllability and existence of mild solutions for Riemann-Liouville fractional stochastic evolution equations with nonlocal conditions of order $1<\alpha<2$, \emph{Fract. Calc. Appl. Anal.}, {\bf 22} (2019), 1086–1112. 
		
		\bibitem{CT} C.C. Travis and G.F. Webb, Cosine families and abstract nonlinear second order differential equations, {\em Acta Math. Hungar.}, {\bf 32} (1978), 75-96.
		
		\bibitem{TCC} C.C. Travis and G.F. Webb, Compactness, regularity, and uniform continuity properties of strongly continuous cosine families, {\em  Houston J. Math.}, {\bf 3} (1977), 555-567.
		
		\bibitem{CT1} C.C. Travis and G.F. Webb, Second order differential equations in Banach space, in {\em Nonlinear Equations in Abstract Spaces}, Academic Press, (1978), 331-361.
		
		\bibitem{TRR} R. Triggiani, Addendum: A note on the lack of exact controllability for mild solutions in Banach spaces, \emph{SIAM J. Control Optim.}, {\bf 18} (1980), 98.
		
		\bibitem{TR} R. Triggiani, A note on the lack of exact controllability for mild solutions in Banach spaces, {\em SIAM J. Control Optim.}, {\bf 15} (1977), 407-411.
		
		\bibitem{VV} V. Vijaykumar, R. Udhayakumar and C. Dineshkumar, Approximate controllability of second order nonlocal neutral differential evolution inclusions, \emph{IMA J. Math. Control Inform.}, (2020), https://doi:10.1093/imamci/dnaa001.
		
		\bibitem{YZ} Z. Yan, Approximate controllability of partial neutral functional differential systems of fractional order with state-dependent delay, {\em Internat. J. Control}, {\bf 85} (2012), 1051–1062.
		
		\bibitem{EZ} E. Zuazua, \emph{Controllability and observability of partial differential equations: some results and open problems}, in Handbook of differential equations: evolutionary equations, {\bf 3} (2007), 527-621.
		
	\end{thebibliography}
\end{document}